\documentclass[final,3p,times]{article}

\usepackage[a4paper,left=25mm,top=25mm,right=20mm,bottom=35mm]{geometry}

\usepackage{authblk}
\usepackage[numbers, sort&compress]{natbib}

\usepackage[colorlinks]{hyperref}
\hypersetup{
	citecolor=blue,
   linkcolor=red,
}

\usepackage{amsmath,amsopn,amsthm,amssymb}

\usepackage[margin=30pt,font=small,labelfont=bf]{caption}

\usepackage{url}
\usepackage{cleveref}
\Crefname{algocf}{Algorithm}{Algorithms}
\usepackage{paralist}
\usepackage{subfigure}
\usepackage{algorithmic}
\usepackage{xcolor,xparse}
\usepackage{dsfont}
\usepackage{mathtools}
\usepackage[ruled,linesnumbered,noend]{algorithm2e}

\DeclareMathSymbol{\minus}{\mathord}{operators}{"2D}    

\newcommand\Ms{\ensuremath{\mathsf{M}}}
\newcommand\Ps{\ensuremath{\mathsf{P}}}

\newcommand\tP{\ensuremath{\theta^{\Ps}}}

\newcommand\Pn{\ensuremath{\mathcal{P}_{n}}}

\newcommand\ol[1]{\ensuremath{\overline{#1}}}
\newcommand\ie{i.\,e.}
\newcommand\eg{e.\,g.}

\newcommand{\Mc}[1]{\ensuremath{\mathcal{#1}}}
\newcommand{\Mods}[2][]{\ensuremath{\mathds{M}_{#1}(#2)}}
\newcommand{\inv}[2]{\ensuremath{[#1 \colon\! #2]}}
\newcommand{\dis}{\ensuremath{\colon\!}}
\newcommand{\Mmax}{\ensuremath{\mathbb{P}_{\max}}}

\definecolor{myred}{RGB}{212,76,76}
\definecolor{mygreen}{RGB}{49,163,84}
\definecolor{myblue}{RGB}{76,155,212}
\definecolor{myyellow}{rgb}{1.0, 0.75, 0.0}
\definecolor{olive}{rgb}{0.5,0.5,0.0}
\newcommand{\Lang}[1]{%
  \ifmmode{%
    \text{\normalfont\textsc{#1}}%
  }%
  \else
  {\normalfont\textsc{#1}}%
  \fi}

\DeclareFontFamily{U}{mathb}{\hyphenchar\font45}
\DeclareFontShape{U}{mathb}{m}{n}{
<-6> mathb5 <6-7> mathb6 <7-8> mathb7
<8-9> mathb8 <9-10> mathb9
<10-12> mathb10 <12-> mathb12
}{}
\DeclareSymbolFont{mathb}{U}{mathb}{m}{n}
\DeclareMathSymbol{\llcurly}{\mathrel}{mathb}{"CE}
\DeclareMathSymbol{\ggcurly}{\mathrel}{mathb}{"CF}

\newcommand{\precdot}{\prec}

\newtheorem{theorem}{Theorem}[section]
\newtheorem{definition}[theorem]{Definition}
\newtheorem{lemma}[theorem]{Lemma}
\newtheorem{corollary}[theorem]{Corollary}
\newtheorem{proposition}[theorem]{Proposition}
\newtheorem{remark}{Remark}

\newtheorem{example}{Example}

\usepackage{ifthen}
\usepackage{tikz}

\usetikzlibrary{fadings}
\newcommand{\tri}[1][0.04]{%
	\begin{tikzpicture}[baseline=-0.5ex,scale=#1]
	\fill[green] (90:4) -- (210:4) -- (-30:4) -- cycle;
	\fill[blue,path fading=west] (90:4) -- (210:4) -- (-30:4) -- cycle;
	\fill[red,path fading=south] (90:4) -- (210:4) -- (-30:4) -- cycle;
	\end{tikzpicture}%
}

\providecommand{\keywords}[1]{\textbf{\textit{Keywords: }} #1}

\title{Complete Edge-Colored Permutation Graphs}

\author[1,*]{Tom Hartmann}

\author[2]{Max Bannach}

\author[3]{Martin Middendorf}

\author[1,4-8]{Peter F. Stadler}    

\author[3]{Nicolas~Wieseke}

\author[9,*]{Marc Hellmuth} 

\affil[1]{Bioinformatics Group, Department of Computer Science \&
    Interdisciplinary Center for Bioinformatics, Universit{\"a}t Leipzig,
    H{\"a}rtelstra{\ss}e~16--18, D-04107 Leipzig, Germany.}

\affil[2]{Institute for Theoretical Computer Science,
    Universit\"at zu L\"ubeck, Ratzeburger Allee~160, D-23562 L\"ubeck,
    Germany.}

\affil[3]{Swarm Intelligence and Complex Systems Group, Faculty of
    Mathematics and Computer Science, University of Leipzig, Augustusplatz
    10, D-04109 Leipzig, Germany.}

\affil[4]{German Centre for Integrative Biodiversity Research
  (iDiv) Halle-Jena-Leipzig, Competence Center for Scalable Data Services
  and Solutions Dresden-Leipzig, Leipzig Research Center for Civilization
  Diseases, and Centre for Biotechnology and Biomedicine at Leipzig
  University at Universit{\"a}t Leipzig}

\affil[5]{Max Planck Institute for Mathematics in the Sciences,
  Inselstra{\ss}e 22, D-04103 Leipzig, Germany} 

\affil[6]{Institute for Theoretical Chemistry, University of Vienna,
  W{\"a}hringerstrasse 17, A-1090 Wien, Austria}

\affil[7]{Facultad de Ciencias, Universidad National de Colombia, Sede
  Bogot{\'a}, Colombia}

\affil[8]{Santa Fe Insitute, 1399 Hyde Park Rd., Santa Fe NM 87501,
  USA}

\affil[9]{School of Computing, University of Leeds, EC Stoner
  Building, Leeds LS2 9JT, UK}

\affil[*]{corresponding author}

\date{\ }

\begin{document}
\sloppy

\maketitle

\abstract{ 
    We introduce the concept of complete edge-colored permutation graphs
      as complete graphs that are the edge-disjoint union of ``classical''
      permutation graphs. We show that a graph $G=(V,E)$ is a complete
    edge-colored permutation graph if and only if each monochromatic
    subgraph of $G$ is a ``classical'' permutation graph and $G$ does not
    contain a triangle with~$3$ different colors.  Using the modular
    decomposition as a framework we demonstrate that complete edge-colored
    permutation graphs are characterized in terms of their strong prime
    modules, which induce also complete edge-colored permutation
    graphs. This leads to an $\Mc{O}(|V|^2)$-time recognition
    algorithm. We show, moreover, that complete edge-colored 
    permutation graphs form a superclass of so-called symbolic ultrametrics
    and that the coloring of such graphs is always a Gallai coloring.
}

%
\section{Introduction}
\label{sec:int}
Permutations model the rearrangement of an ordered sequence of
objects. Thus they play an important role in card
shuffling~\cite{Aldous1986}, comparative
genomics~\cite{Bhatia2018,Fertin2009,Hartmann2018}, and combinatorial
optimization \cite{Bona2016}. The effect of a permutation can be
illustrated in a \emph{permutation graph}~\cite{Pnueli:1971} that contains
the elements as vertices and that connects two vertices with an edge if the
corresponding elements are reversed by the permutation.
       
The graph class of permutation graphs is of considerable theoretical
interest~\cite{Brandstadt1999,Golumbic2004} because many computationally
intractable problems can be solved efficiently on permutation graphs. For
example, \Lang{hamiltonian-cycle} (given a graph, is there a cycle
containing every vertex?)~\cite{Deogun1994}, \Lang{feedback-vertex-set}
(can we transform the graph into a forest by deleting at most $k$
vertices?)~\cite{Brandstadt1985}, or \Lang{maximum-independent-set} (can we
find an independent set of largest possible size for a given
graph?)~\cite{Kim:1990} can be solved efficiently on permutation graphs.
Furthermore, permutation graphs can be recognized in linear
time~\cite{Crespelle2010,Mcconnell1999} and several characterizations of
permutation graphs have been
established~\cite{Brandstadt1999,Gallai1967,Dushnik1941,BFR:72,mohring1985algorithmic}.

Multiple permutations over the same set of elements can be represented in a
\emph{single} edge-colored graph by taking the edge-union of the
corresponding permutation graphs and assigning to each edge a color that
uniquely identifies the underlying permutation. In particular, if the edge
sets of the underlying permutation graphs are disjoint, then, each edge
obtains a unique label (or ``color'') in the resulting graph.  For certain
sets of permutation graphs, this procedure ends up with a \emph{complete
  edge-colored graph}, a class of graphs that has received
  considerable attention as \emph{symmetric 2-structures}
  \cite{ER1:90,ER2:90,EHPR:96,ehrenfeucht1999theory,Hellmuth:16a}.

In this article we study the reverse direction: \emph{Given a complete
  edge-colored graph, does it represent a set of permutations?} In a graph
theoretic sense, we are interested in the structure of graphs that can be
generated as superpositions of permutation graphs. For instance, we may
ask: Do the induced subgraphs of such graphs have certain interesting
properties? Are there forbidden structures? Which graph classes are a
superset of such colored permutation graphs?

\paragraph{Contribution I: A Characterization of Complete Edge-Colored
  Permutation Graphs}
Our first contribution is a collection of characterizations of
\emph{complete edge-colored permutation graphs} (graph of the just sketched
kind; a formal definition is given in Section~\ref{sec:permG}). We first
show that these graphs are \emph{hereditary}, \ie, each induced subgraph
of a complete edge-colored permutation graph is again a complete
edge-colored permutation graph. 

Furthermore, we provide a characterization that is closely related to
Gallai colorings: A complete edge-colored graph is a complete edge-colored
permutation graph if and only if all its monochromatic subgraphs are 
``classical'' permutation graphs and it does not contain a \emph{rainbow
  triangle}, \ie, a triangle with three distinct edge colors.

Finally, we provide two characterizations in terms of the \emph{modular
  decomposition:} A complete edge-colored graph is a complete edge-colored
permutation graph if and only if the quotient graph of each strong module
is a complete edge-colored permutation graph.  Moreover, a complete
edge-colored graph is a complete edge-colored permutation graph if and only
if the quotient graph of each strong prime module is a complete
edge-colored permutation graph.

\paragraph{Contribution II: Recognition of Complete Edge-Colored
  Permutation Graphs}
We prove that complete edge-colored permutation graphs $G=(V,E)$ can be
recognized in $\Mc{O}(|V|^2)$-time and, in the affirmative case, the
underlying permutations can be constructed within the same time complexity.

\paragraph{Contribution III: Connection to other Graph Classes}
We provide a classification of the class of complete edge-colored
permutation graphs with respect to other graph classes. We show that the
coloring of each complete edge-colored permutation graph is a Gallai
coloring, \ie, it is a complete edge-colored graph that does not contain a
rainbow triangle.  Furthermore, we prove that every graph representation of
a \emph{symbolic ultrametric}, \ie, a graph that is based on sets of
certain symmetric binary relations, is also a complete edge-colored
permutation graph. Moreover, we show how symbolic ultrametrics, complete
edge-colored graphs, and so-called separable permutations are related.
        
\paragraph{Related Work}
Permutation graphs were already characterized in 1976 in terms of forbidden
subgraphs~\cite{Gallai1967} and in a number of other characterizations, see
for instance \cite{Brandstadt1999,Golumbic2004}. Although many
computationally intractable problems become tractable on permutation
graphs, this is not the case for some coloring problems such as the problem
to determine the achromatic number~\cite{Bodlaender1989} or the cochromatic
number~\cite{Wagner1984}. There are also many problems for which the
complexity on permutations graphs is still unknown, for instance
determining the edge search number~\cite{Golovach2012}.  A general overview
on permutation graphs can be found in \cite{Brandstadt1999,Golumbic2004}.
        
The idea to investigate graphs whose edge set can be decomposed into a set
of permutation graphs is not new; in fact, not necessarily complete
versions of those graphs were studied by Golumbic \emph{et al.} in the
1980s \cite{golumbic1983comparability}.  It was shown that the complement
of such a graph is a comparability graph, \ie, a graph that corresponds to
a strict partially ordered set.  In addition, the authors investigated the
problem to find a minimum number of permutations whose edge-union forms a
given graph.  For more information on those graphs, the reader is referred
to \cite{Golumbic2004}.
        
The investigation of complete edge-colored graphs without rainbow triangles
also has a long history. A complete graph is said to admit a \emph{Gallai
  coloring} with $k$ colors, if its edges can be colored with $k$ colors
without creating a rainbow triangle~\cite{Gyarfas2004}. It is an active
field of research to study when such a coloring may
exist~\cite{Gyarfas2019,ball2007colored} and what properties such colorings
have~\cite{Balogh2018,Bastos2019,Bastos2018}.  For example, a Gallai
coloring of the $K_n$, \ie, the complete graph with $n\in \mathbb{N}$
vertices, contains at most $n-1$ colors \cite{Erdoes1973}.  Another
well-known property is that those graphs can be obtained by substituting
complete graphs with Gallai colorings into vertices of 2-edge-colored complete
graphs \cite{Gallai1967,Gyarfas2004,Cameron1997}.  A survey on complete
edge-colored graphs that admit a Gallai coloring can be found \eg\ in
\cite{Fujita2010,Kano2008,Gyarfas2011}.

\emph{Symbolic ultrametrics}~\cite{Bocker1998} are symmetric binary
relations that are closely related to vertex-colored trees.  They were
recently used in phylogenomics for the characterization of homology
relations between
genes~\cite{hellmuth2013orthology,hellmuth2015phylogenomics}.  The graph
representation of a symbolic ultrametric can be considered as an
edge-colored undirected graph in which each vertex corresponds to a leaf in
the corresponding vertex-colored tree.  Two vertices $x$ and $y$ of this
graph are connected by an edge of a specific color if the last common
ancestor of the corresponding leafs $x$ and $y$ in the tree has that
specific color.  Building upon the results
of~\cite{Bocker1998,semplephylogenetics}, these graphs were studied
intensively in recent years, see for
instance~\cite{hellmuth2013orthology,hellmuth2015phylogenomics,hellmuth2016sequence,hellmuth2018tree,lafond2014orthology}
and references therein. They are of central interest in phylogenomics as
their topology can be represented by an event-annotated phylogenetic tree
of the given genetic sequences~\cite{Bocker1998,semplephylogenetics}.  It
has recently been shown that the graph representation of a symbolic
ultrametric is a generalization of
\emph{cographs}~\cite{hellmuth2013orthology}, which are graphs that do not
contain an induced path on four vertices~\cite{CORNEIL1981163}. In
addition, the authors have characterized these graphs in terms of forbidden
subgraphs.

Many results that are utilized in this contribution, \ie, the principles of
modular decomposition of complete edge-colored graphs, are based on the
theory of $2$-structures which was first introduced by Ehrenfeucht and
Rozenberg~\cite{ER1:90,ER2:90}.  The notion of a $2$-structure can be seen
as a generalization of the notion of a graph and, thus, it provides a
convenient framework for studying graphs.  In particular, $2$-structures
facilitate the deduction of strong decomposition theorems for graphs, since
they allow to represent graphs hierarchically as
trees~\cite{ehrenfeucht1990partial}.  For more information and surveys on
$2$-structures see \cite{EHPR:96,ehrenfeucht1999theory}.

For an overview on modular graph decompositions we refer to the survey by
Habib and Paul~\cite{Habib2010}, which summarizes the algorithmic ideas and
techniques that arose from these decompositions. Various results are known
that connect permutation graphs and modular decompositions, see for
instance~\cite{Crespelle2010,Bergeron2008,Capelle2002}.
        
\paragraph{Organization of this Contribution}
In Section~\ref{sec:prems} we provide formal definitions concerning
permutations, graphs, and their modular decomposition.  The class of
permutation graphs and
complete edge-colored permutation graphs are described in
Section~\ref{sec:permG} and characterized in Section~\ref{sec:chara}.  In
\Cref{sec:recognition} we show that complete edge-colored permutation
graphs can be recognized in $\Mc{O}(|V|^2)$-time.
Section~\ref{sec:graph_classes} investigates the connection between the
class of complete edge-colored permutation graphs, symbolic ultrametrics,
and Gallai colorings.  In Section~\ref{sec:conclusion} a summary of the
paper is given and avenues for future work are outlined.


\section{Preliminaries: Permutations, Graphs, Modular Decomposition}
\label{sec:prems}

In this section, we provide the necessary formal notion and respective
useful results of permutations, graphs, and modular graph decomposition.

\paragraph{\textbf{Permutations}}
In the following we write $\inv{1}{n}\coloneqq \{1,\dots,n\}$.  A
\emph{permutation} $\pi\colon \inv{1}{n}\to \inv{1}{n}$ of length
$n\in\mathbb{N}$ is a bijective map that assigns to each element
$i\in \inv{1}{n}$ a unique element $\pi(i)\in \inv{1}{n}$.  By slight abuse
of notation, we represent a permutation $\pi$ as the sequence
$(\pi(1),\dots ,\pi(n))$. We denote the \emph{length} of
$\pi = (\pi(1),\dots ,\pi(n))$ by $|\pi|=n$ and the set of all permutations
of length $n$ by~$\Pn$.

For every $\pi \in \Pn$ there is a unique permutation $\pi^{-1} \in \Pn$
that is called \emph{inverse permutation} (of $\pi$) defined by
$\pi^{-1}(j) = i$ if and only if $\pi(i)=j$.  We denote with $\ol{\pi}$ the
\emph{reversed} permutation in which the order of all elements with respect
to $\pi$ is reversed: $\ol{\pi}(i) = \pi(n+1-i)$ for all $i\in \inv{1}{n}$.
A \emph{subsequence} (of $\pi$) is a sequence $(\pi(i_1), \dots, \pi(i_k))$
with $1 \leq i_1 \leq \dots \leq i_k \leq n$.  The permutation
$\iota\coloneqq(1, \dots ,n)$ is called \emph{identity}.

\begin{example}
  Consider the permutation $\pi=(1,5,2,4,7,3,6)$ of length $7$. Its inverse
  permutation is $\pi^{-1}=(1,3,6,4,2,7,5)$ and its reversed permutation is
  $\ol{\pi}=(6,3,7,4,2,5,1)$.  The sequence $(1,2,4,7,6)$ is a subsequence
  of~$\pi$.
\end{example}

\paragraph{\textbf{Binary Relations and Strict Total Orders}}
A binary relation $R\subseteq V\times V$ on $V$ is a \emph{strict total
  order} on $V$ if

\begin{enumerate}
\item[(R1)] $R$ is trichotomous, \ie, for all $x,y\in V$ we have either
  $(x,y)\in R$, $(y,x)\in R$ or $x=y$ and
\item[(R2)] $R$ is transitive, \ie, for all $x,y,z\in V$ with
  $(x,y),(y,z)\in R$ it holds that $(x,z)\in R$.
\end{enumerate}
Note that (R1) is equivalent to claiming that $R$ is irreflexive,
asymmetric, and that all distinct elements $x,y\in V$ are in relation $R$.

\paragraph{\textbf{Graphs and Colorings}}
An \emph{(undirected) graph} $G$ is an ordered pair $(V,E)$ consisting of a
non-empty, finite set of \emph{vertices}~$V(G)=V$ and a set of \emph{edges}
$E(G)=E\subseteq \binom{V}{2}$, where $\binom{V}{2}$ denotes the set of all
2-element subsets of $V$. Observe that this definition explicitly excludes
self loops and parallel edges.  If $G=(V,\binom{V}{2})$, then $G$ is called
\emph{complete} and denoted by $K_{|V|}$.

For an edge $e=\{u, v\}$ of $G$, the vertices $u$ and $v$ are called
\emph{adjacent} and $e$ is said to be \emph{incident} with $u$ and $v$.
Let $v \in V$ be a vertex of $G$.  The \emph{degree} of $v$ is the number
of edges in $G$ that are incident with $v$.  The \emph{complement} of $G$
is the graph $\ol{G}$ with vertex set $V$ and edge set
$\binom{V(G)}{2} \setminus E$.  A sequence of vertices
$S=(v_1, \dots, v_m)$ of $G$ is called a \emph{walk} (in $G$) if
$\{v_i,v_{i+1}\}\in E$ for all $i\in\inv{1}{m-1}$. A walk is called a
\emph{path} if it does not contain a vertex more than once.
A graph $G$ is \emph{connected} if for any two vertices $u,v\in V$ there
exists a walk from $u$ to $v$, otherwise $G$ is called \emph{disconnected}.
		
A graph $H=(W,F)$ is a \emph{subgraph} of $G=(V,E)$ if $W\subseteq V$ and
$F\subseteq E$.  A subgraph $H=(W,F)$ (of $G$) is called \emph{induced},
denoted by $G[W]$, if for all $u,v\in W$ it holds that $\{u,v\} \in E$
implies $\{u,v\} \in F$.
A \emph{connected component} of a graph is a connected subgraph that is
maximal with respect to inclusion.  A subset of vertices $V$ of a graph $G$
is an \emph{independent set} (resp.\ a \emph{clique}) if for all distinct
$u,v\in V$ it holds that $\{u,v\} \notin E$ (resp.\ $\{u,v\} \in E$).
		
A \emph{tree} $T=(V,E)$ is a graph in which any two vertices are connected
by exactly one path.  A \emph{rooted tree} $T$ is a tree with one
distinguished vertex $\rho \in V$ called the \emph{root}.  The
\emph{leaves} $L \subseteq V$ of $T$ is the set of all vertices that are
distinct from the root and have degree 1.  Let $T=(V,E)$ be a rooted tree
with leaf set $L$ and root $\rho$. Then, vertex $v$ is called a
\emph{descendant} of vertex $u$ and $u$ an \emph{ancestor} of $v$
if $u$ lies on the unique path from $v$ to the root $\rho$.  
The \emph{children} of an inner vertex $v$ are its direct descendants, \ie,
vertices~$w$ with $\{v,w\} \in E$ and $w$ is a descendant of $v$.  
In this case, the vertex $v$ is called the \emph{parent} of $w$.

\begin{definition}[$k$-edge-colored graph] 
  A \emph{$k$-edge-colored} graph is a graph $G=(V,E)$ together with a
  surjective map $c\colon E\to \inv{1}{k}$ assigning a unique color~$c(e)$
  to \emph{each} edge $e\in E$.
\end{definition}

The map~$c$ is also called $k$-edge-coloring of the edges of~$G$.
A $k$-edge-colored graph is also called \emph{edge-colored} graph if the
number~$k$ of colors remains unspecified.  By slight abuse of 
notation, we also call graphs $G=(V,E)$ with $E=\emptyset$ \emph{edge-colored
graphs}. Thus, the graph $K_1=(\{v\},\emptyset)$ is an edge-colored graph. 
Using this notation, every graph $G=(V,E)$  
	can be considered as a 1-edge-colored graph.

Given a $k$-edge-colored graph with non-empty edge set and a
map~$c$, the edge set $E$ can be partitioned into~$k$ non-empty subsets
$E_1,\dots E_k$, where each $E_i$ contains all edges $e$ with $c(e)=i$. We
call such a partition of $E$ into the sets $E_1,\dots,E_k$ an
\emph{(edge-)coloring} of $G=(V,E)$ and write $(V,E_1,\dots ,E_k)$ instead
of $(V,E)$.  

For a subgraph $H$ of a $k$-edge-colored graph $G$, we always
assume that each edge $e$ in~$H$ retains the color~$c(e)$ that it has
in~$G$, making $H$ a $k'$-edge-colored graph with $k'\leq k$.  The \emph{$i$-th
  monochromatic} subgraph of a $k$-edge-colored graph $G$ is the subgraph
$G_{|i}\coloneqq(V,E_i)$, where $E_i$ contains all edges with color~$i$.

\begin{remark}
  Many of the $k$-edge-colored graphs considered here are complete graphs.
  Note, this is not a restriction in general, since each $k$-edge-colored graph
  can be transformed into a complete ($k+1$)-edge-colored graph by
  assigning a new color $k+1$ to all non-edges.
\end{remark}

\begin{definition}[Rainbow triangle]\label{def:rainbowT}
  Let $G=(V,E_1,\dots ,E_k)$ be a $k$-edge-colored graph and
  $\{u,v,w\}\subseteq V$. The induced subgraph $G[\{u,v,w\}]$ is called a
  \emph{rainbow triangle}, denoted by $\tri_{uvw}$, if $\{u,v\}$,
  $\{u,w\}$, and $\{v,w\}$ have pairwise distinct colors.
\end{definition}
	
In the literature, various notions of rainbow triangles can be found.
Examples are the notions~\emph{multicolored triangles}
\cite{gyarfas2010gallai} and \emph{tricolored
  triangles}~\cite{Gyarfas2004}.

A quite useful property of complete $k$-edge-colored graphs without rainbow
triangles is summarized in the following
\begin{proposition}[\cite{Gyarfas2004}, Lemma A]
  \label{lem:gyarfas_disconnected}
  For every complete $k$-edge-colored graph $G$ with $k\in\mathbb{N}_{\geq 3}$
  that does not contain a rainbow triangle as induced subgraph there exists
  a color $i\in\inv{1}{k}$ such that the $i$-th monochromatic subgraph
  $G_{|i}$ of $G$ is disconnected.
\end{proposition}

\begin{example}\label{exa:colored_graph}
  Consider the complete graph $G=(\{a,b,c,d,e,f\},E_1,E_2,E_3,E_4)$ that is
  illustrated in \Cref{fig:coloredGraph}.  Graph~$G$ is complete and
  $4$-edge-colored. Moreover, $G$ contains the rainbow triangle $\tri_{def}$,
  although the $1$-st monochromatic subgraph
  $G_{|1}=(\{a,b,c,d,e,f\},\{\{a,d\},\{b,d\},\{c,d\}\})$ of~$G$ is
  disconnected. This, in particular, shows that the converse of
  \Cref{lem:gyarfas_disconnected} is not always satisfied.
\end{example}

\begin{figure}
  \centering
  \includegraphics[width=0.35\linewidth]{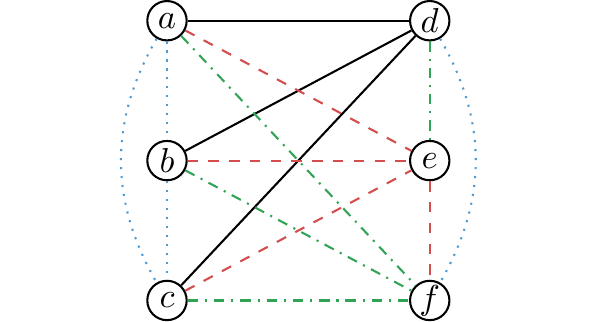}
  \caption{A complete $4$-edge-colored graph
    $G=(V=\{a,b,c,d,e,f\},E_1,E_2,E_3,E_4)$ as used in
    \Cref{exa:colored_graph}.  The edge-coloring of $G$ is illustrated
    using different edge styles: $E_1$ (black solid edges), $E_2$ (red
    dashed edges), $E_3$ (blue dotted edges), and $E_4$ (green dash-dotted
    edges). There are four non-trivial modules in~$G$, namely
    $M=\{a,b,c\}$, $\{a,b\}$, $\{b,c\}$, and $\{a,c\}$. However, $M$ is the
    only non-trivial strong module of~$G$. Thus,
    $\Mmax(V) = \{\{a,b,c\}, \{d\}, \{e\},\{f\}\}$.  }
    \label{fig:coloredGraph}	
\end{figure}

\paragraph{\textbf{Modular Decomposition}}
A \emph{module} of a complete $k$-edge-colored graph $G=(V,E_1,\dots ,E_k)$ is a
subset $M \subseteq V$ such that for every $u \in M$ and
$v \in V\setminus M$ it holds that if $\{u,v\}\in E_i$ then
$\{w,v\} \in E_i$ for all $w\in M$. The set~$V$, the empty set $\emptyset$,
and the singleton sets $\{v_1\},\dots ,\{v_{|V|}\}$ are called
\emph{trivial} modules. The set of all modules of $G$ is denoted by
$\Mods{G}$.  Examples of modules can be found in \Cref{fig:coloredGraph}
and \Cref{exa:modularDeco}.  A module $M \in \Mods{G}$ is called
\emph{strong} if for any other module $M'\in\Mods{G}$ we have
$M\subseteq M'$, $M'\subseteq M$, or $M\cap M'=\emptyset$.  In the
following, all modules considered here are non-empty unless explicitly
stated otherwise.

The set of strong modules $\Mods[str]{G}$ is called \emph{modular
  decomposition} of~$G$.  Whereas there may be exponentially many modules,
the size of the set of strong modules is~$\Mc{O}(|V|)$~\cite{EHMS:94}. In
particular, since strong modules do not overlap by definition, it holds
that \Mods[str]{G} forms a hierarchy which gives rise to a unique tree
representation $T_G$ of $G$: the \emph{modular decomposition tree} of $G$
or \emph{inclusion tree} of
\Mods[str]{G}~\cite{ehrenfeucht1999theory,Hellmuth:16a}. The vertices of
$T_G$ are identified with the elements of \Mods[str]{G}.  Adjacency in
$T_G$ is defined by the maximal proper inclusion relation, \ie, there is 
an edge $\{M, M'\}$ between $M, M' \in \Mods[str]{G}$  
if and only if $M \subset M'$ (resp.\ $M' \subset M$) and there is no $M''\in \Mods[str]{G}$ such that
$M \subset M'' \subset M'$ (resp.\ $M' \subset M'' \subset M$). The root of $T_G$ is $V$ and the leaves are the
singletons $\{v\}$ with $v \in V$.

We denote with $\Mmax(M)$ the set of all children
$M_1,\dots ,M_n \in \Mods[str]{G}$ of a strong module $M \in \Mods[str]{G}$
in the inclusion tree $T_G$. $\Mmax(M)$ is also called (maximal) modular
partition of $G[M]$ due to the fact that $\Mmax(M)$ forms a partition of
$M$ in which each element is strong and maximal with respect to inclusion
\cite{Hellmuth:16a}.

Modular decomposition trees allow us to extend the ``usual'' graph 
isomorphism to complete edge-colored graphs as follows: 
Two graphs $G=(V,E)$ and $G'=(V',E')$ are \emph{isomorphic}, if there is a
bijection $\phi \colon\! V \rightarrow V'$ such that for all vertices
$u,v \in V$ it holds that $\{u,v\}\in E$ if and only if
$\{\phi(u),\phi(v)\}\in E'$. 
Two complete $k$-edge-colored graphs $G=(V,E_1,\dots ,E_k)$
and $G'=(V',E'_1,\dots ,E'_k)$ are \emph{isomorphic}, denoted by
$G \simeq G'$, if there is a bijection
$\pi\dis \inv{1}{k}\to \inv{1}{k}$ such that for all
$i\in\inv{1}{k}$ the $i$-th monochromatic
subgraph of $G$ and the $\pi(i)$-th monochromatic subgraph of $G'$ are isomorphic. 
As shown in \cite{ER1:90,ER2:90}, two complete edge-colored graphs
are isomorphic precisely if their modular decomposition trees 
are isomorphic. This result, in particular, implies that
two $k$-edge-colored graphs $G$ and $G'$ are isomorphic (using the 
respective bijections $\phi$ and $\pi$ as above) 
if and only if,
for all subsets $W\subseteq V$, the induced subgraph $G[W]$ and 
$G'[W']$ with $W'=\{\phi(v)\mid v\in W\}$ are isomorphic 
(using the same  bijection $\pi$ as used for the colors of $G$ and $G'$).

For the definition of so-called quotient graphs we need the following

\begin{lemma}[\cite{ER1:90}, Lemma 4.11]\label{lem:sameColor}
  Let $M$ and $M'$ be two disjoint modules of a complete $k$-edge-colored 
  graph $G=(V,E_1,\dots ,E_k)$.  Then, there is a color $i\in \inv{1}{k}$
  such for all edges $\{u,v\}\in E(G)$ with $u\in M$ and $v\in M'$ it holds
  that $\{u,v\}\in E_i$.
\end{lemma}

\Cref{lem:sameColor} allows us to define the \emph{quotient graph}
$G/\Mmax(V)$ of a given complete $k$-edge-colored graph $G=(V,E_1,\dots ,E_k)$
and a modular partition $\Mmax(V)=\{M_1,\dots ,M_n\}$ of $G$ as follows:
$G/\Mmax(V)=(V',E'_1,\dots ,E'_k)$ has as vertex set $V'=\Mmax(V)$ and
$\{M_i,M_j\}\in E'_h$, $i\neq j$ if and only if there are vertices
$u\in M_i$, $v\in M_j$ such that $\{u,v\}\in E_h$.  Moreover, since all
edges between distinct modules in $G$ must have the same color, we can
naturally extend the edge-coloring of $G$ to its quotient graph
$G/\Mmax(V)$, that is, an edge~$\{M_i,M_j\}$ obtains color $c(\{u,v\})$ for
some $u\in M_i$ and $v\in M_j$.

\begin{remark}\label{rem:subgraph}
  By construction, for every strong module $M$ of $G$, the quotient graph
  $G[M]/\Mmax(M)$ is isomorphic to every induced subgraph $G[N]$ of $G$,
  where $N$ contains exactly one vertex from each child module
  $M_i\in \Mmax(M) = \{M_1, \dots, M_k\}$.  Moreover, if $G$ is a complete
  graph, then $G[N]$ is complete as well.  In this case, $G[M]/\Mmax(M)$ is
  always a complete graph.
\end{remark}

As shown in \cite{ER2:90,EHPR:96}, the quotient graph $G[M]/\Mmax(M)$ for a
strong module $M$ of a complete $k$-edge-colored graph~$G$ is well-defined and
satisfies exactly one of the following conditions:
\begin{enumerate}
\item[1)] $G[M]/\Mmax(M)$ contains at least two vertices and is a complete 1-edge-colored
  graph.  In this case, $M$ is called a \emph{series} module.
\item[2)] $G[M]/\Mmax(M)$ does not satisfy~(1).  In this case, $M$ is
  called a \emph{prime} module.
\end{enumerate}
For non-complete graphs there may exist a third type of strong module
called parallel (when $G[M]$ is disconnected)~\cite{MOHRING:1984}, a case
that cannot occur here.  We are aware of the fact that the definition of
the type of strong modules (series, prime) slightly differs from the one
provided, \eg, in \cite{ER1:90,ER2:90,Hellmuth:16a}, in which the type of
strong modules is defined in terms of $G[M]$ rather than on
$G[M]/\Mmax(M)$. To be more precise, in \cite{ER1:90,ER2:90,Hellmuth:16a} a
strong module $M$ is ``prime'' whenever $G[M]$ consists of trivial modules
only and $M$ is ``series'' whenever $G[M]$ is a complete 1-edge-colored
graph. In this case, however, the complete 1-edge-colored graph $K_2$ is
``prime'' and ``series'' at the same time.  Moreover, the complete
2-edge-colored graph $K_3$ with vertices $a,b,c$ and coloring
$c(\{a,c\})= c(\{b,c\}) \neq c(\{a,b\})$ would, in this definition, neither
be ``prime'' nor ``series''. Nevertheless, the quotient graphs always fall
into one of these two categories.  Hence, in order to simplify and
streamline the notation, we use this definition in terms of
$G[M]/\Mmax(M)$.

\begin{definition}[Primitive graph]
  A (complete $k$-edge-colored) graph $G$ is called \emph{primitive} if and
  only if G contains only trivial modules. 
\end{definition}

For an example, consider again the~$K_2$ or the complete 2-edge-colored~$K_4$
that contains a monochromatic subgraph isomorphic to an induced path on
four vertices. Both graphs are primitive.  On the contrary, every
2-edge-colored~$K_3$ is not primitive as it contains always a non-trivial module
of size~$2$.

A useful property of quotient graphs is provided in the following
\begin{lemma}[{\cite[Thm.\ 2.17]{EHPR:96}}]
  \label{lem:PrimeImpliesPrimitive}
  Let $M$ be a strong prime module of a complete $k$-edge-colored graph
  $G$. Then $G[M]/\Mmax(M)$ is primitive.
\end{lemma}

\begin{example}
  \label{exa:modularDeco}
  Consider the complete $4$-edge-colored graph
  $G =(\inv{1}{7}, E_1,E_2,E_3,E_4)$ as is illustrated in
  \Cref{fig:modularDecomposition}~(a) with edge sets:
  \begin{align*}
    E_1 &= \big\{\,\{1,6\},\{2,6\},\{3,6\},\{3,7\},
        &E_2&=\big\{\,\{3,4\},\{3,5\}\,\big\},\\
        &\qquad\{4,6\},\{4,7\},\{5,6\},\{5,7\}\,\big\},\\[.75ex]
    E_3 &=\binom{\inv{1}{7}}{2}\setminus(E_1\cup E_2 \cup E_4),
        &E_4&=\big\{\,\{1,2\}\,\big\}.
  \end{align*}
  The modular decomposition of $G$ is
  $\Mods[str]{G}=\{\inv{1}{7},
  \{1,2\},\{3,4,5\},\{4,5\},\{1\},\dots,\{7\}\}$ and the maximal strong
  modular partition of the module $\inv{1}{7}$ is
  $\Mmax(\inv{1}{7})=\{\{1,2\},\{6\},\{3,4,5\},\{7\}\}$.
  \Cref{fig:modularDecomposition}~(b) shows the modular decomposition tree
  of $G$.  \Cref{fig:modularDecomposition}~(c) and
  \Cref{fig:modularDecomposition}~(d) illustrate the quotient graphs
  $G/\Mmax(\inv{1}{7})$ and $G[\{4,5\}]/\Mmax(\{4,5\})$, respectively.  It can
  easily be verified that the module $\inv{1}{7}$ is prime, whereas the
  modules $\{1,2\}$, $\{4,5\}$, and $\{3,4,5\}$ are series.
\end{example}
	
\begin{figure}
  \centering \includegraphics[width=0.9\linewidth]{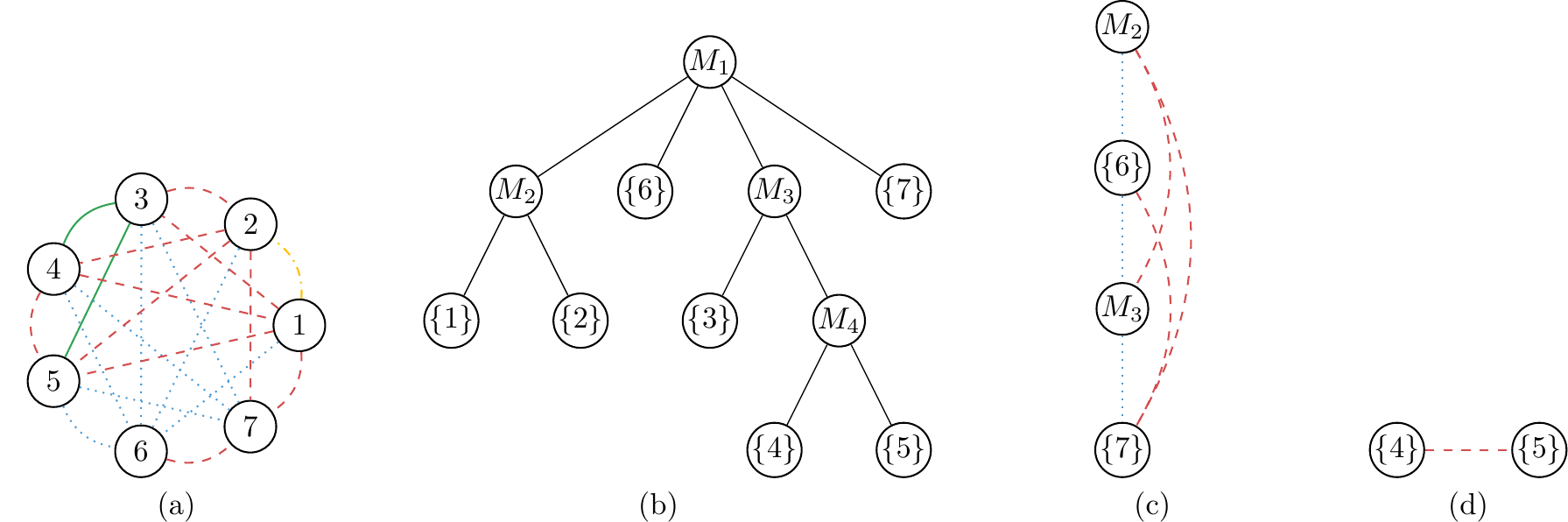}
  \caption{Panel~(a) illustrates the complete $4$-edge-colored graph $G$ that is
    defined in \Cref{exa:modularDeco}.  The edge-coloring of $G$ is
    illustrated using different edge styles: $E_1$ (blue dotted edges),
    $E_2$ (green solid edges), $E_3$ (red dashed edges), and $E_4$ (yellow
    dash-dotted edge).  The modular decomposition tree $T_G$ of $G$ is
    shown in Panel~(b), where $M_1=\inv{1}{7}$, $M_2=\{1,2\}$,
    $M_3=\{3,4,5\}$, and $M_4=\{4,5\}$.  Panels~(c) and (d) illustrate the
    quotient graph $G[M_1]/\Mmax(M_1)$ and $G[M_4]/\Mmax(M_4)$,
    respectively.  }
  \label{fig:modularDecomposition}
\end{figure}

Interestingly, every complete primitive graph that contains at least three
colors must contain a rainbow triangle:
\begin{lemma}\label{lem:prime_rainbow}
  Let $k\in\mathbb{N}_{\geq 3}$ and $G$ be a complete $k$-edge-colored graph.
  If $G$ is primitive, then $G$ contains a rainbow triangle.
\end{lemma}
\begin{proof}
  The proof proceeds by contraposition: We show that every complete
  $k$-edge-colored graph $G$ with $k\in\mathbb{N}_{\geq 3}$ that does not
  contain a rainbow triangle as induced subgraph contains a non-trivial
  module, \ie, $G$ is not primitive.
	
  Let $G=(V,E_1,\dots ,E_k)$ be a complete $k$-edge-colored graph with
  $k\in\mathbb{N}_{\geq 3}$ that does not contain a rainbow triangle.  By
  \Cref{lem:gyarfas_disconnected} there is a color $i\in\inv{1}{k}$ such
  that $G_{|i}$ is disconnected.  Let $M_1,\dots ,M_m$, $m\geq 2$ be the
  connected components of $G_{|i}$.  Since $G$ is a $k$-edge-colored graph
  with $k\geq 3$, at least one connected component, say $M_1$, contains an
  edge $\{u,v\}$ of color $i$.  Furthermore, let $x\in V\setminus M_1$ be a
  vertex that is not in $M_1$. Note that neither $\{u,x\}$ nor $\{v,x\}$
  can have color $i$ since $x\notin M_1$.  Therefore, $\{u,x\}$ must have
  some color $j\neq i$.  Since $G$ does not contain rainbow triangles, the
  edge $\{v,x\}$ must have color $j$ as well. Repeating the latter
  arguments along the edges in the spanning tree of $G_{|i}[M_1]$ implies
  that all vertices in $M_1$ must be connected to $x$ via an edge of color
  $j$.
  The latter, in particular, remains true for all vertices in
  $V\setminus M_1$, that is, all edges leading from vertices in $M_1$ to a
  fixed vertex $x\in V\setminus M_1$ have the same color.  Hence, $M_1$ is
  a non-trivial module in $G$ which implies that $G$ cannot be primitive.
\end{proof}

\section{Definition and Basic Properties of Complete Edge-Colored Permutation Graphs}
\label{sec:permG}

In this section we introduce the formal definitions and provide some useful
results for \emph{(complete edge-colored) permutation} graphs.  Recall that
a permutation is a bijection from $\inv{1}{n}$ to itself. Intuitively, the
vertices of a permutation graph are the elements of $\inv{1}{n}$ and the
edges indicate whether the order of two elements is ``reversed'' by the
permutation.  For technical reasons that become clear later it is more
convenient to use a labeling $\ell$ that assigns to each vertex $v$ a
natural number rather than to define the vertex set $V$ directly as the set
$\inv{1}{|V|}$.  To this end, we need the following:
\begin{definition}[Labeling]
  A \emph{labeling} $\ell$ of a graph $G=(V,E)$ is a bijective map
  $\ell\dis V \rightarrow \inv{1}{|V|}$ that associates each vertex $v$
  with a unique natural number $\ell(v)\in \inv{1}{|V|}$.
\end{definition}
Using a labeling, permutation graphs~\cite{Pnueli:1971} can be 
defined as follows.
\begin{definition}[Permutation Graph \cite{Pnueli:1971,Brandstadt1999,Koh2007}]
  A graph $G$ is called a \emph{(simple) permutation graph} if there exists
  a labeling $\ell$ and a permutation $\pi=(\pi(1),\dots,\pi(|V|))$ such
  that for all $u,v\in V$ we have:
  \begin{equation*}
    \{u,v\} \in E\Longleftrightarrow
    \ell(u)>\ell(v)
    \text{ and }
    \pi^{-1}(\ell(u))<\pi^{-1}(\ell(v)).
  \end{equation*}
  In this case, we say that $G$ is a \emph{(simple) permutation graph} of
  $\pi$.  If the labeling $\ell$ is specified, we may also write that $G$
  together with $\ell$, or simply $(G,\ell)$, is a (simple)
  \emph{permutation graph} (of $\pi$).  
\end{definition}
A permutation graph of $\pi=(1,5,2,4,7,3,6)$ is illustrated
in~\Cref{fig:permutationGraph}~(a).

\begin{figure}
  \centering
  \includegraphics[width=0.7\linewidth]{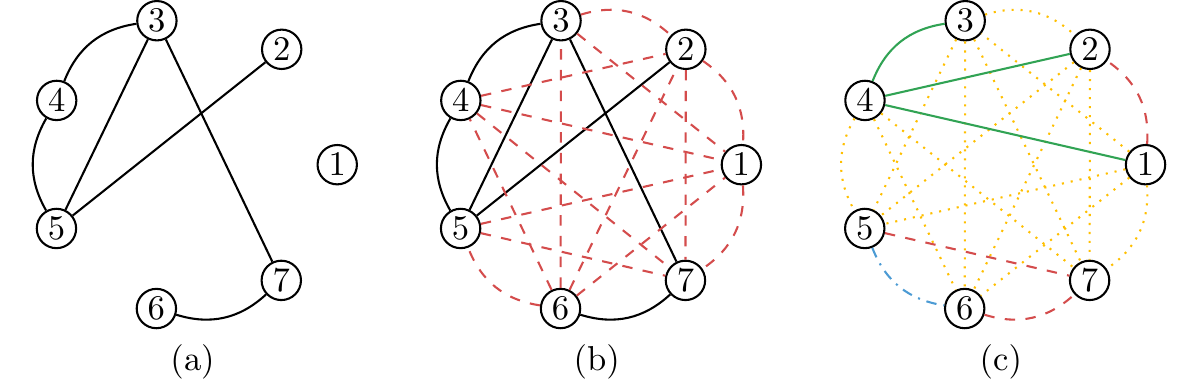}
  \caption{Three different edge-colored graphs $(V,E)$ together with
    labeling $\ell\colon V\to \inv{1}{7}$ that is implicitly given in the
    drawing by writing $\ell(v)$ on each vertex $v\in V$.  Panel~(a)
    illustrates the simple permutation graph~$G$ of~$\pi=(1,5,2,4,7,3,6)$.
    For example, $G$ contains the edge $\{3,7\}$ since $7>3$ and
    $\pi^{-1}(7)<\pi^{-1}(3)$.  Panel~(b) shows the complete $2$-edge-colored
    permutation graph of $\pi$ and $\ol{\pi}=(6,3,7,4,2,5,1)$.  Panel~(c)
    shows the complete $4$-edge-colored permutation graph of the permutations
    $(2~1~3~4~7~5~6)$ (red dashed edges), $(4~1~2~3~5~6~7)$ (green solid
    edges), $(1~2~3~4~6~5~7)$ (blue dash-dotted edge), and
    $(5~6~7~3~1~2~4)$ (yellow dotted edges).  }
  \label{fig:permutationGraph}
\end{figure}

The property of being a simple permutation graph is hereditary:
\begin{lemma}[\cite{Crespelle2010}] \label{lem:simplePerm-hereditary}
The family of permutation graphs is \emph{hereditary}, that is, 
every induced subgraph of a permutation graph is a permutation graph.
\end{lemma}

The following known \namecref{lem:asc_desc_seq} relates subsequences
of a permutation to independent sets and cliques in the corresponding
permutation graph.
\begin{proposition}[{\cite[p.~159]{Golumbic2004}}]\label{lem:asc_desc_seq}
  Let $(G,\ell)$ be a simple permutation graph of some permutation
  $\pi \in \Pn$.  The following statements are true for any subsequence
  $S=(\pi(i_1),\dots, \pi(i_k))$ of $\pi$ with
  $1 \leq i_1 \leq \dots \leq i_k \leq n$:
  \begin{enumerate}
  \item[(i)] $S$ is ascending if and only if
    $\{\,\ell^{-1}(\pi(i))\mid i_1 \leq i \leq i_k\,\}$ is an independent
    set in~$G$.
  \item[(ii)] $S$ is descending if and only if
    $\{\,\ell^{-1}(\pi(i))\mid i_1 \leq i \leq i_k\,\}$ is a clique in~$G$.
  \end{enumerate}
\end{proposition}

By \Cref{lem:asc_desc_seq}, a simple permutation graph is edgeless (resp.,
complete) if and only if the corresponding permutation is the identity
permutation (resp., the reversed identity permutation).

\begin{corollary}
  \label{cor:indepentetSet_iota}
  Let $G=(V,E)$ be the simple permutation graph of a permutation $\pi$. The
  following statements are true:
  \begin{enumerate}
  \item[(i)]
    $\hbox to
    0pt{$E=\emptyset$}\phantom{E=\binom{V}{2}}\Longleftrightarrow\pi=\iota$.
  \item[(ii)] $E=\binom{V}{2}\Longleftrightarrow\pi=\ol{\iota}$.
  \end{enumerate}
\end{corollary}
\begin{proof}
  Follows directly from \Cref{lem:asc_desc_seq}.
\end{proof}

Another well-known property of simple permutation graphs is summarized in
the following \namecref{prop:perm_complement}. It states that the complement
of a simple permutation graph of a permutation $\pi$ is isomorphic to the
simple permutation graph of the reversed permutation $\ol{\pi}$ of $\pi$.
\begin{proposition}[\cite{Koh2007}, Proposition~2.2]
  \label{prop:perm_complement}
  Assume that $(G,\ell)$ is a simple permutation graph of
  $\pi\in \Mc{P}_{|V|}$. Then $(\ol{G},\ell)$ is the simple permutation
  graph of $\ol{\pi}$.
\end{proposition}

Instead of considering simple permutation graphs (which represent a
		single permutation) we are interested in \emph{complete $k$-edge-colored permutation graphs}
		(which represent a set of $k$ permutations).

\begin{definition}[Complete $k$-edge-colored Permutation Graph]
  Let $G=(V,E_1,\dots,E_k)$ be a complete and $k$-edge-colored graph.  Then, $G$
  is a \emph{complete $k$-edge-colored permutation graph}   if there is a labeling
  $\ell\dis V \rightarrow \inv{1}{|V|}$ and $k$ permutations
  $\pi_1,\dots,\pi_k$ of length $|V|$ such that for all $i\in\inv{1}{k}$
  the $i$-th monochromatic labeled subgraph $(G_{|i},\ell)$ is a simple
  permutation graph of $\pi_i$, \ie, for all $i\in\inv{1}{k}$ and all
  $u,v\in V$ we have:
  \begin{equation*}
    \{u,v\} \in E_i\Longleftrightarrow
    \ell(u)>\ell(v)
    \text{ and }
    \pi_i^{-1}(\ell(u))<\pi_i^{-1}(\ell(v)).
  \end{equation*}
  In this case, we say that $G$ is a \emph{complete $k$-edge-colored permutation
    graph} or simply, complete edge-colored permutation graph (of
  $\pi_1,\dots,\pi_k$).  If the labeling $\ell$ is specified, we may also
  write that $G$ together with $\ell$, or simply $(G,\ell)$, is complete
  $k$-edge-colored \emph{permutation graph} (of $\pi_1,\dots,\pi_k$).
\end{definition}
In~\Cref{fig:permutationGraph}~(c) a complete $4$-edge-colored permutation graph
is illustrated.
	
Note that for a complete $k$-edge-colored permutation graph $G\not\simeq K_1$ of
permutations~$\pi_1,\dots,\pi_k$ all elements $E_1, \ldots, E_k$ are
non-empty, since the edge-coloring of $G$ is surjective.  Thus, by
\Cref{cor:indepentetSet_iota}, we have that $\pi_i \neq \iota$ for all
$i\in\inv{1}{k}$.  \Cref{cor:indepentetSet_iota} also implies that for all
$n>1$ there is only one complete\linebreak $1$-edge-colored permutation graph on
$|V|=n$ vertices, namely the complete graph $K_n$ which is the simple
permutation graph of $\ol{\iota}$.  Clearly, the definition of $1$-edge-colored permutation graphs coincides with the
definition of simple permutation graphs. 

\begin{corollary}\label{cor:two_colored_one_label}
  Let $G$ be a complete $k$-edge-colored graph with $k\in\{1,2\}$.
  Assume that $(G_{|1},\ell)$ is a simple permutation graph for
  $\pi$.  Then $(G,\ell)$ is a complete $k$-edge-colored permutation graph
  for $\pi$ (and $\ol{\pi}$).
\end{corollary}
\begin{proof}
  Let $G$ be a complete $k$-edge-colored graph with $k\in\{1,2\}$.  Furthermore,
  let $(G_{|1},\ell)$ be a simple permutation graph for $\pi$.  The case
  $k=1$ is trivial and, thus, assume that $G=(V,E_1,E_2)$ is 2-edge-colored.  In
  this case, $E_1$ and $E_2$ are both non-empty.  By
  \Cref{cor:indepentetSet_iota}, we obtain that $\pi\neq\iota$.  Moreover,
  $\pi$ cannot be $\ol\iota$ as otherwise $G$ would be complete
  $1$-edge-colored.  Since $(G_{|1},\ell)$ is a simple permutation graph for $\pi$, we
  have that $\{u,v\}\in E_1$ if and only if $\ell(u) > \ell(v)$ and
  $\pi^{-1}(\ell(u))<\pi^{-1}(\ell(v))$.
	
  By \Cref{prop:perm_complement}, $(\ol{G_{|1}},\ell)$ is a simple permutation
  graph of $\ol{\pi}$. Hence, we have that $\{u,v\} \in E(\ol{G_{|1}})$ if
  and only if $\ell(u) > \ell(v)$ and
  $\ol{\pi}^{-1}(\ell(u))<\ol{\pi}^{-1}(\ell(v))$.  Clearly, this is if and
  only if ${\pi}^{-1}(\ell(u))>{\pi}^{-1}(\ell(v))$ for all non-edge
  $\{u,v\}$ in $G_{|1}$, which completes the proof.
\end{proof}

By \Cref{cor:two_colored_one_label}, every simple permutation graph~$G$ that is
neither edgeless nor complete corresponds to a complete $2$-edge-colored
permutation graph by interpreting the non-edges of $G$ as edges with some
new color, see also \Cref{fig:permutationGraph}~(a) and
\Cref{fig:permutationGraph}~(b) for illustrative examples.  Note, however,
that \Cref{cor:two_colored_one_label} cannot easily be extended to
$k$-edge-colored (possibly non-complete) graphs where each monochromatic
subgraph is a simple permutation graph, cf.\ \Cref{fig:2vs3} for an
example.

\begin{figure}
  \centering
  \includegraphics[width=0.8\linewidth]{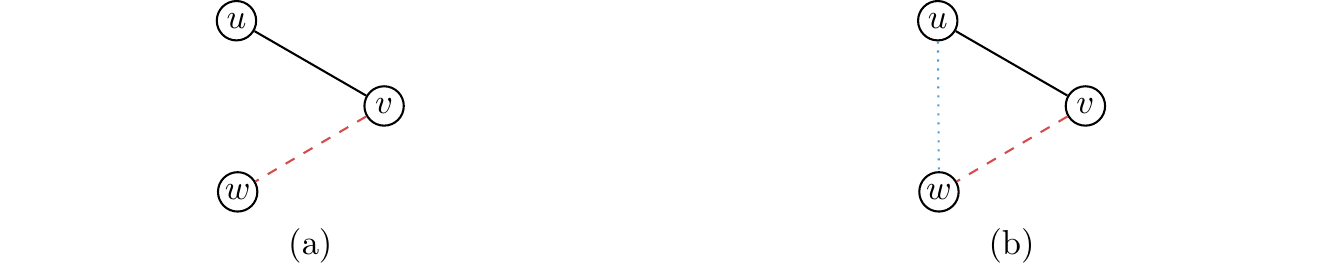}
  \caption{(a) A non-complete 2-edge-colored graph $G$ with color~$1$ (black
    solid edge) and color~2 (red dashed edge) is shown.  Each of the two
    monochromatic subgraphs $G_{|1}$ and $G_{|2}$ of $G$ is isomorphic to
    the graph consisting of three vertices, a single edge, and an isolated
    vertex.  Let us fix the labeling $\ell$ of $G$ by putting $\ell(u)=1$,
    $\ell(v)=2$ and $\ell(w)=3$. Now, it is easy to verify that
    $(G_{|1},\ell)$ and $(G_{|2},\ell)$ are simple permutation graphs for
    $\pi_1=(2,1,3)$ and $\pi_2=(1,3,2)$, respectively.  For both
    permutations, even the non-edge $\{u,w\}$ would be ``explained'', since
    $\ell(u)=1<\ell(w)=3$ and $\pi_i^{-1}(\ell(u))< \pi_i^{-1}(\ell(w))$,
    $i=1,2$. \newline We emphasize, however, that this is not the
    definition of edge-colored permutation graphs we use.  We explicitly
    consider here only \emph{complete} edge-colored graphs.  Hence, we may
    extend~$G$ to the complete 3-edge-colored graph $G'$ shown in~(b).  However,
    $G'$ is not a complete edge-colored permutation graph, since it
    contains a rainbow triangle (cf.\ \Cref{def:rainbowT} and
    \Cref{pro:characterization}).}
  \label{fig:2vs3}
\end{figure}

The following two lemmas show that complete edge-colored permutation graphs
contain forbidden substructures.

\begin{lemma}\label{lem:forbidden_subgraph}
  Let $G$ be a complete edge-colored graph, $\ell$ be a labeling of $G$,
  and $u,v,w$ be three distinct vertices of $G$ with
  $\ell(u)<\ell(v)<\ell(w)$. Suppose that the colors of edges $\{u,v\}$ and
  $\{v,w\}$ are equal and different from the color of edge $\{u,w\}$, then
  $(G,\ell)$ cannot be a complete edge-colored permutation graph.
\end{lemma}
\begin{proof}
  Let $G$, $\ell$, and $u,v,w$ be as specified in the lemma. 
	Clearly, $G=(V,E_1,\dots,E_k)$ for some $k\geq 2$. 
	Assume that the edges $\{u,v\}$ and   $\{v,w\}$ have color $i\in \inv{1}{k}$
	and edge $\{u,w\}$ has color~$j\in \inv{1}{k}$, $j\neq i$.

  Assume, for contradiction, that $(G,\ell)$ is a complete $k$-edge-colored
  permutation graph of $\pi_1,\dots,\pi_k$. By assumption we have
  $\ell(u)<\ell(v)<\ell(w)$.  Hence, in order to realize the colors of
  edges $\{u,v\}$ and $\{v,w\}$ it must hold that
  $\pi_i^{-1}(\ell(u))> \pi_i^{-1}(\ell(v))$ and
  $\pi_i^{-1}(\ell(v))>\pi_i^{-1}(\ell(w))$, respectively.  Hence, we have
  $\pi_i^{-1}(\ell(u))> \pi_i^{-1}(\ell(w))$.  This together with
  $\ell(u)<\ell(w)$ implies that $\{u,w\}\in E_i$; a contradiction.
\end{proof}

\begin{lemma}\label{lem:rainbow_nonprime2}
  If a complete edge-colored graph $G=(V,E_1,\dots,E_k)$ contains a rainbow triangle, then 
  $G=(V,E_1,\dots,E_k)$ cannot be a complete edge-colored permutation graph.
\end{lemma}
\begin{proof}
  Let $G=(V,E_1,\dots,E_k)$ be a complete $k$-edge-colored graph.  Since $G$
  contains, by assumption, a rainbow triangle, we have 
  $k\geq 3$ and thus, $|V|\geq 3$.
	
  Let $\tri_{uvw}$ be a rainbow triangle of $G$ with colors
  $\{i_1,i_2,i_3\} \subseteq \inv{1}{k}$.  Without loss of generality, we
  may assume that $\{u,w\}$, $\{u,v\}$ and $\{v,w\}$ have color $i_1$,
  $i_2$ and $i_3$, respectively.
  
  Assume, for contradiction, that $G$ together with some labeling
  $\ell \dis V \rightarrow \inv{1}{|V|}$ is a complete $k$-edge-colored
  permutation graph for some permutations $\pi_1,\dots,\pi_{k}$.  Observe
  that $\ell$ implies a strict order on $\{\ell(u),\ell(v),\ell(w)\}$, as
  $u$, $v$, and $w$ are all distinct. Without loss of generality, let this
  order be $\ell(u)>\ell(v)>\ell(w)$.

  Since $\{i_1,i_2,i_3\} \subseteq \inv{1}{k}$, it follows that
  $\pi_{i_1},\pi_{i_2},\pi_{i_3} \in \{\pi_1,\dots,\pi_k\}$.  Since the
  color of $\{u,w\}$ is $i_1$, it holds that $\ell(u)>\ell(w)$ and
  $\pi_{i_1}^{-1}(\ell(u)) < \pi_{i_1}^{-1}(\ell(w))$.  Now consider the
  vertices $v$ and $w$.  We either have (i)
  $\pi_{i_1}^{-1}(\ell(v)) < \pi_{i_1}^{-1}(\ell(w))$ or (ii)
  $\pi_{i_1}^{-1}(\ell(v)) > \pi_{i_1}^{-1}(\ell(w))$.  In Case (i), the
  edge $\{v,w\}$ would have color $i_1$ because $\ell(v)>\ell(w)$; a
  contradiction to the assumption that $\tri_{uvw}$ is a rainbow triangle.
  However, in Case (ii), we have
  $\pi_{i_1}^{-1}(\ell(u)) <
  \pi_{i_1}^{-1}(\ell(w))<\pi_{i_1}^{-1}(\ell(v))$.  This together with
  $\ell(u)>\ell(v)$ implies that the edge $\{u,v\}$ would have color~$i_1$.
  Again we derive a contradiction to the assumption that $\tri_{uvw}$ is a
  rainbow triangle.  Consequently, $(G,\ell)$ cannot be a complete
  $k$-edge-colored permutation graph for any labeling $\ell$ of $G$.  Hence,
  $(G,\ell)$ cannot be a complete edge-colored permutation graph at all.
\end{proof}

\section{Characterization of Complete Edge-Colored Permutation Graphs}\label{sec:chara}

We provide here the main results of this contribution.
\begin{theorem} \label{thm:characterization_all_k_PG} Suppose that $G$ is a
  complete edge-colored graph.  Then the
  following statements are equivalent:

\begin{itemize}
\item[(i)] $G$ is a complete edge-colored permutation graph, which
	is, by definition, if and only if there exists a labeling~$\ell$ of~$G$ such that for each monochromatic subgraph
	$G_{|i}$ of~$G$ it holds that $(G_{|i},\ell)$ is a simple permutation
	graph.
\item[(ii)] Each induced subgraph of $G$ is a complete edge-colored
  permutation graph, \ie, the property of being a complete edge-colored
  permutation graph is hereditary.
\item[(iii)] The quotient graph $G[M]/\Mmax[M]$ of each strong module $M$
  in the modular decomposition of $G$ is a complete edge-colored
  permutation graph
\item[(iv)] The quotient graph $G[M]/\Mmax[M]$ of each strong prime module
  $M$ in the modular decomposition of $G$ is a complete edge-colored
  permutation graph. In particular, the quotient graph of each strong prime
  module on at least three vertices must be $2$-edge-colored.
\item[(v)] Each monochromatic subgraph of $G$ is together with some
  labeling (possibly different from the labeling $\ell$ as chosen in Item (i)) 
	a simple permutation graph and
  $G$ does not contain a rainbow triangle.
\end{itemize}
\end{theorem}

It is easy to see that, Item \emph{(ii)} implies Item \emph{(i)}
    (since $G$ is an induced subgraph of itself) and Item \emph{(iii)} implies 
		the first part of Item \emph{(iv)}.

In what follows, we prove \Cref{thm:characterization_all_k_PG} and verify
the individual items of \Cref{thm:characterization_all_k_PG} in
\Crefrange{section:item3}{sec:mono}.  
In particular, we show that all Items~\emph{(ii)} - \emph{(v)} are equivalent to Item \emph{(i)}. 
The reader might have already observed that there is an alternative 
avenue to show the equivalence between Items~\emph{(i)}, \emph{(iii)}, and \emph{(iv)}, 
since  Item \emph{(iii)} trivially implies the first part of Item \emph{(iv)}. 
This observation suggests to prove the implication ``$(i)\Rightarrow (iii)$''
and ``$(iv)\Rightarrow (i)$''. Nevertheless, we treat the equivalence 
between Items~\emph{(i)} and \emph{(iii)} a well as \emph{(i)} and \emph{(iv)} in separate steps.
The reason is simple: the proof of implication ``$(iii)\Rightarrow(i)$''
is constructive and used to design our efficient recognition 
algorithm in \Cref{sec:recognition}. In particular,
this result allows us to provide a simple proof to show that $(iv)$ implies $(i)$.

For the sake of completeness,
\Cref{section:mainproof} summarizes all results to prove
\Cref{thm:characterization_all_k_PG}.

\subsection{An Induced Subgraph Characterization}\label{section:item3}

In this section, we show that the property of being a complete edge-colored
permutation graph is hereditary and that complete edge-colored permutation
graphs are characterized by this property. That is, we show the equivalence of \Cref{thm:characterization_all_k_PG}~(i) 
and \Cref{thm:characterization_all_k_PG}~(ii).

Clearly, it is not hard to see that \Cref{thm:characterization_all_k_PG}~(ii) trivially implies \Cref{thm:characterization_all_k_PG}~(i). 
However, for the sake of completeness, we summarize this result in the following

\begin{lemma}\label{cor:i_iii-b}
	If every induced subgraph of a graph $G$ is a complete edge-colored permutation graph, then $G$ is a complete edge-colored permutation graph. 
\end{lemma}  
\begin{proof}
	The claim is a consequence of the simple fact that $G=(V,E)$ and the induced subgraph $G[V]$ are identical.
\end{proof}

We now show that the converse of \Cref{cor:i_iii-b} is valid as well.
\begin{proposition}\label{pro:i_iii}
  If $G=(V,E)$ is a complete edge-colored permutation graph, then, for all
  non-empty subsets $M\subseteq V$, the induced subgraph $G[M]$ is a
  complete edge-colored permutation graph.  Hence, the property of being a
  complete edge-colored permutation graph is hereditary.
\end{proposition}
\begin{proof}
  Let $G=(V,E_1,\dots,E_k)$ together with a labeling $\ell$ be a complete
  $k$-edge-colored permutation graph for the permutations
  $\pi_1,\dots,\pi_k$.  The cases $|M|=1$ as well as $k=1$ are trivial and
  thus, we assume that $m=|M|>1$ as well as $k>1$.
	
  Now, we utilize $\ell$ and $\pi_1,\dots ,\pi_k$ to construct a labeling
  $\hat{\ell}$ and permutations $\hat{\pi}_1,\dots ,\hat{\pi}_{k}$ such
  that $(G[M],\hat{\ell})$ is a complete $k'$-edge-colored permutation graph for
  $k'\leq k$ permutations from $\{\hat{\pi}_1,\dots ,\hat{\pi}_{k}\}$.

  First, let $\Ms\coloneqq\{\,\ell(v)\mid v\in M\,\}$ be the set of all
  labels assigned to the vertices in $M$. Moreover, let
  $\Ps_i\coloneqq \{\pi_i^{-1}(\ell(v))\mid v\in M\,\}$ for all
  $i\in \inv{1}{k}$.  Now, for every $i\in \inv{1}{k}$ we obtain the
  well-defined bijective map
  \begin{align*}
    \sigma_i \colon  \mathsf \Ps_i&\to \Ms \\ 
    j&\mapsto \pi_i(j). 
  \end{align*}
  We now relabel the labels in $\Ms$ by defining the map
  $\breve{\ell}\colon\Ms \rightarrow \inv{1}{m}$ via
  $\breve{\ell}(i) < \breve{\ell}(j)\Leftrightarrow i<j$ for all distinct
  $i,j\in \Ms$.  We use $\breve{\ell}$ to define a labeling $\hat \ell$ of
  $G[M]$ as
  \begin{align*}
    \hat \ell \colon  M &\to \inv{1}{m}\\ 
    j&\mapsto \breve{\ell}(\ell(j)).
  \end{align*}
  Note that $\hat \ell$ is bijective and well-defined.

  In order to transform each $\sigma_i$ to a permutation
  $\hat{\pi}_i\colon \inv{1}{m} \to \inv{1}{m}$, we additionally provide
  bijective maps $\tP_1,\dots \tP_k$ by defining for all $r\in \inv{1}{k}$
  the map $\tP_r\colon \inv{1}{m} \rightarrow \Ps_r$ as
  $\tP_r(i) < \tP_r(j)\Leftrightarrow i<j$ for all distinct $i,j\in \Ps_r$.
  Note that $\breve{\ell}$ as well as $\tP_1,\dots, \tP_k$ are unique and
  well-defined since $\Ms$ as well as $\Ps_1,\dots, \Ps_k$ are totally
  ordered sets on~$m$ elements.
	
  \begin{figure}[t]
    \centering \includegraphics[width=0.9\linewidth]{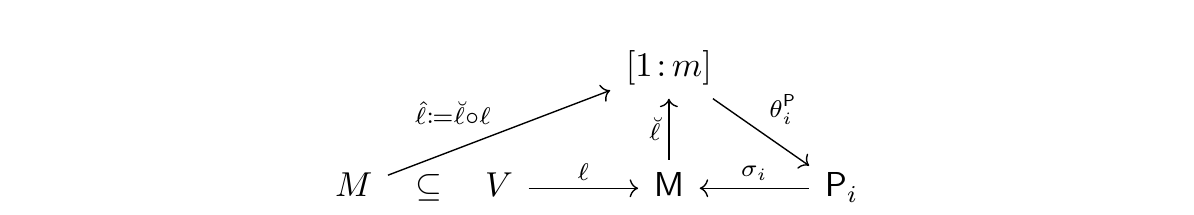}
    \caption{Diagram showing the maps
      $\sigma_i, \ell, \breve{\ell}, \hat \ell$, and $\tP_i$ as used in the
      proof of~\Cref{pro:i_iii}.}
    \label{fig:proof_prop_5}
  \end{figure}
  See \Cref{fig:proof_prop_5} for an illustration of the maps
  $\sigma_i, \breve{\ell}, \hat\ell$, and $\tP_i$ that are used in this
  proof.
	
  To derive the aforementioned permutations
  $\hat{\pi}_i\colon \inv{1}{m} \to \inv{1}{m}$, we simply define
  $\hat{\pi}_i$ as
  $\hat{\pi}_i\coloneqq(\breve{\ell}(\sigma_i(\tP_i(1))),\dots
  ,\breve{\ell}(\sigma_i(\tP_i(m))))$ for all $i\in\inv{1}{k}$. In other
  words, for all $j\in \inv{1}{m}$ and all $i\in\inv{1}{k}$ we have
  $\hat{\pi}_i(j) = \breve{\ell}(\sigma_i(\tP_i(j)))$.
			
  Now $G[M]$ is a complete graph that retains the colors of the edges as
  provided in $G$.  However, observe that $G[M]$ may have less than~$k$
  colors. Let $I=\{i_1,\dots,i_{k'}\} \subseteq \inv{1}{k}$ be the set of
  all colors on edges in $G[M]$.  Hence,
  $G[M]=(M,\hat{E}_{i_1},\ldots,\hat{E}_{i_{k'}})$ is a complete edge-colored
  graph.
	
  We continue to show that $(G[M],\hat\ell)$ is a complete $|I|$-edge-colored
  permutation graph for $\hat{\pi}_i$, $i\in I$.  To this end, let
  $u,v \in M$ be two vertices and without loss of generality assume that
  $\ell(u)>\ell(v)$.  Since we have not changed the relative order of the
  labels in $\Ms$ under $\breve \ell$, we have
  $\breve \ell(\ell(u) ) > \breve \ell(\ell(v))$ which implies
  $\hat \ell (u) > \hat \ell (v)$.
	
  First assume that $\{u,v\}\in \hat{E}_i$ for some $i\in I$.  Since $G[M]$
  is an induced subgraph of $G$ it holds that $\{u,v\} \in E_i$.  Since
  $\ell(u) > \ell(v)$ and $\{u,v\}\in E_i$ we have
  $\pi^{-1}_i(\ell(u)) < \pi^{-1}_i(\ell(v))$.  Since
  $\hat \ell(u) > \hat \ell(v)$, we must show that
  $\hat \pi^{-1}_i(\hat \ell(u)) < \hat\pi^{-1}_i(\hat \ell(v))$.
		
  For better readability, we put $f\coloneqq (\tP_i)^{-1}$.  By
  construction, $\hat \pi_i = \breve{\ell}\circ \sigma_i\circ \tP_i$ and
  thus, $\hat \pi^{-1}_i = f\circ \sigma_i^{-1}\circ \breve{\ell}^{-1}$.  Now
  let us examine $\hat \pi^{-1}_i(\hat\ell(x))$ for some $x\in M$.  First,
  $\breve \ell^{-1}(\hat \ell(x))=\breve \ell^{-1}(\breve \ell(\ell(x))) =
  \ell(x)$.  Hence, $\sigma_i^{-1}(\breve \ell^{-1}(\hat \ell(x)))$ reduces
  to $\sigma_i^{-1}(\ell(x)) = \pi_i^{-1}(\ell(x))$.
		
  Hence, $\hat \pi^{-1}_i(\hat \ell(u)) = f(\pi_i^{-1}(\ell(u)))$ and
  $\hat \pi^{-1}_i(\hat \ell(v)) = f(\pi_i^{-1}(\ell(v)))$.  Note that the
  order $\pi^{-1}_i(\ell(u)) < \pi^{-1}_i(\ell(v))$ is, by construction,
  preserved under $ f$, that is,
  \begin{equation*}
    \hat \pi^{-1}_i(\hat \ell(u)) = f(\pi_i^{-1}(\ell(u)))<
    f(\pi_i^{-1}(\ell(v))) = \hat \pi^{-1}_i(\hat \ell(v)).
  \end{equation*}
  If $\{u,v\}\notin \hat{E}_i$, we can use an analogous argumentation to
  conclude that
  $\hat \pi^{-1}_i(\hat \ell(u)) > \hat\pi^{-1}_i(\hat \ell(v))$.
		
  In summary, we have shown that $\{u,v\}\in \hat{E}_i$ with
  $\hat \ell(u) > \hat \ell(v)$ if and only if
  $\hat \pi^{-1}_i(\hat \ell(u)) < \hat\pi^{-1}_i(\hat \ell(v))$.  Hence,
  $(G[M],\hat\ell)$ is a complete $|I|$-edge-colored permutation graph for
  $\hat{\pi}_i$, $i\in I$.
\end{proof}

\subsection{A Characterization in Terms of Quotient Graphs} \label{sec:strongM}
	 
In this section, we study the quotient graphs of strong modules of complete
edge-colored permutation graphs. In particular, we show that a graph is a
complete edge-colored permutation graph if and only if the quotient graph
of each strong module is a complete edge-colored permutation graph. 
In other words, we show the equivalence of \Cref{thm:characterization_all_k_PG}~(i) and \Cref{thm:characterization_all_k_PG}~(iii).

\begin{proposition}\label{lem:colored_perm_iff2}
  Let $G$ be a complete edge-colored permutation graph. Then, the quotient
  graph $G[M]/\Mmax(M)$ of each strong module $M$ in the modular
  decomposition of $G$ is a complete edge-colored permutation graph.
\end{proposition}
\begin{proof}
  Let $G=(V,E)$ be a complete edge-colored permutation graph and let
  $M \subseteq V$ be a strong module of $G$.
	
  If $M=\{v\}$ for some $v\in V$, then $G[M]/\Mmax(M)$ is the single vertex
  graph that is, by definition, a complete edge-colored permutation graph.
		
  Now assume that $|M|>1$ and thus, $\Mmax(M) = \{M_1, \ldots , M_m\}$,
  $m>1$.  Let $N\subseteq M$ be such that $N$ contains exactly one vertex
  from every strong module $M_i\in \Mmax(M)$.
  
  Since $G[N]$ is an induced subgraph of the complete edge-colored
  permutation graph $G$, it follows by \Cref{pro:i_iii} that $G[N]$ is a
  complete edge-colored permutation graph.  By \Cref{rem:subgraph}, $G[N]$
  is isomorphic to $G[M]/\Mmax(M)$. Hence, the statement immediately
  follows for the quotient graph $G[M]/\Mmax(M)$.
\end{proof}

\begin{figure}
  \centering \includegraphics[width=0.85\linewidth]{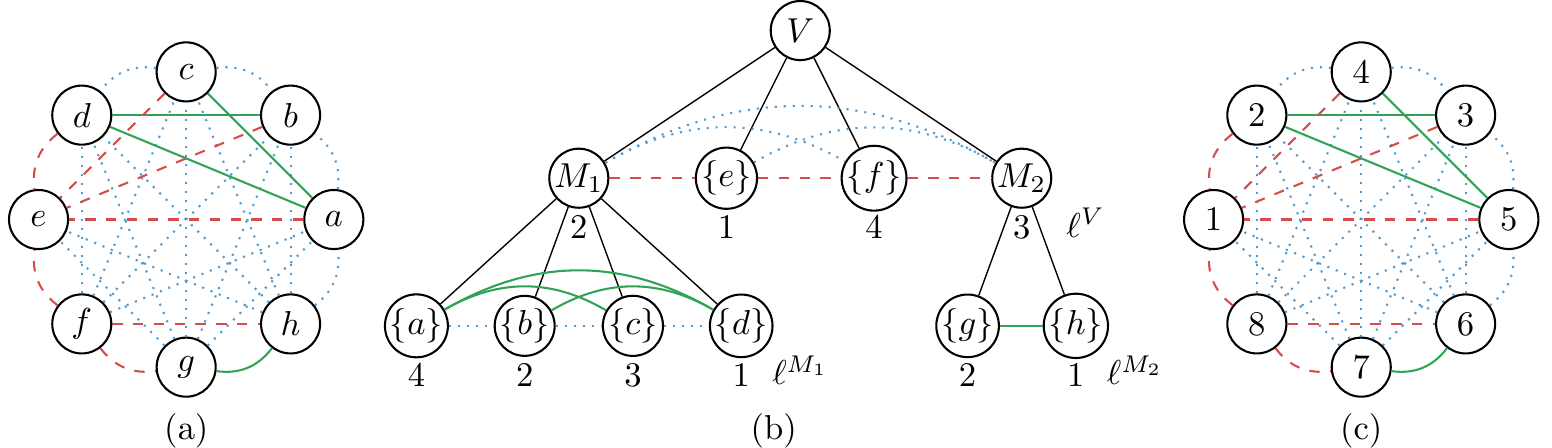}
  \caption{A complete $3$-edge-colored graph~$G$ without a specified labeling
    (Panel~(a)).  The modular decomposition tree $T_G$ of~$G$ together with
    the quotient graph for each strong module (Panel~(b)); and~$G$ together
    with labeling $\ell_{\precdot}$ such that $(G,\ell_{\precdot})$ is
    complete $3$-edge-colored permutation graph of $\pi_1 = (1,3,5,2,4,7,6,8)$,
    $\pi_2 = (6,7,1,8,4,2,5,3)$, and $\pi_3=(2,3,4,5,8,1,6,7)$
    (Panel~(c)). \newline Let us denote the colors green-lined, blue-dotted
    and red-dashed by $1$, $2$ and $3$, respectively.  Using the strict
    total order $\precdot$ defined on the labeling $\ell^M$ of the quotient
    graphs $G[M]/\Mmax(M)$ (as specified with the integers written below
    the vertices in $T_G$), we obtain
    $e \precdot d \precdot b \precdot c \precdot a \precdot h \precdot g
    \precdot f$.  The resulting labeling $\ell_{\precdot}$ is shown in
    Panel~(c). \newline The strict total orders for each color
    $\llcurly_1$, $\llcurly_2$ and $\llcurly_3$ defined on
    $\ell_{\precdot}$ are
    $e \llcurly_1 b \llcurly_1 a \llcurly_1 d \llcurly_1 c \llcurly_1 g
    \llcurly_1 h \llcurly_1 f$;
    $h \llcurly_2 g \llcurly_2 e \llcurly_2 f \llcurly_2 c \llcurly_2 d
    \llcurly_2 a \llcurly_2 b$; and
    $ d \llcurly_3 b \llcurly_3 c \llcurly_3 a \llcurly_3 f \llcurly_3 e
    \llcurly_3 h \llcurly_3 g$.  Following the proof of
    \Cref{lem:colored_perm_iff2-b}, we obtain $\pi_1, \pi_2$ and $\pi_3$
    from $\llcurly_1$, $\llcurly_2$ and $\llcurly_3$, respectively.  }
  \label{fig:working_example}
\end{figure}

In what follows we show that the converse of \Cref{lem:colored_perm_iff2}
is valid as well.  To this end, we first specify an ordering $\precdot$ of
the vertices of $G$ based on the labelings of the quotient graphs of strong
modules of $G$ (cf.\ \Cref{def:order1}). The ordering~$\precdot$ can then
be used to establish a specific labeling $\ell_{\precdot}$ of $G$. In order
to show that $(G,\ell_{\precdot})$ is a complete edge-colored permutation
graph we need to provide the respective permutations.  We construct these
permutations by employing additional orderings $\llcurly_i$ of $V$ for each
of the individual colors~$i$ of $G$ (cf.\ \Cref{def:order2}).  See also
\Cref{fig:working_example} for an illustrative example.
		 
However, to establish these results we first need some further notation.
Let $G$ be a graph with vertex set $V$ and $T_G$ be the modular
decomposition tree of~$G$.  For $u,v\in V$, we denote with $M^{u,v}$ the
inclusion-minimal strong module of $G$ that contains $u$ and $v$.  Note
that if $u=v$, then we have $M^{u,u} = \{u\}$.  Let us suppose that $u$ and
$v$ are distinct, in which case $\{u\},\{v\}, M^{u,v} \in \Mods[str]{G}$.
By definition, $\{u\}$ and $\{v\}$ are leaves of $T_G$ and $M^{u,v}$ is an
inner vertex of $T_G$.  Note, the module $M^{u,v}$ is the common ancestor
of $\{u\}$ and $\{v\}$ that is located farthest from the root of~$T_G$,
also known as lowest common ancestor~\cite{Aho1976,Harel1984}.  Since $T_G$
is a tree, it is not hard to see that $M^{u,v}$ is uniquely defined.  By
construction there exist (unique) distinct strong modules
$M^{u,v}_u, M^{u,v}_v \in \Mmax(M^{u,v})$ with $u \in M^{u,v}_u$ and
$v \in M^{u,v}_v$.  Note that $M^{u,v}_u$ and $M^{u,v}_v$ are the child
modules of $M^{u,v}$ in $T_G$ that have $\{u\}$ and $\{v\}$ as descendant,
respectively.  Since $T_G$ is a tree, it follows for all distinct
$u,v,w \in V$ that $1\leq|\{M^{u,v}, M^{v,w}, M^{u,w}\}|\leq 2$.

\begin{definition}\label{def:order1}
  Let $G=(V,E)$ be a complete edge-colored graph. Moreover, for all
  $M\in \Mods[str]{G}$ let $\ell^M$ denote a labeling of the quotient graph
  $G[M]/\Mmax(M)$.  Then, $\precdot$ denotes the binary relation on $V$
  that satisfies for all distinct $u,v\in V$
  \begin{equation*}
    u \precdot v \Leftrightarrow \ell^{M^{u,v}}(M^{u,v}_u) <
    \ell^{M^{u,v}}(M^{u,v}_v).
  \end{equation*}
  Moreover, let $\ell_{\precdot}$ be the map
  $\ell_{\precdot}\dis V \rightarrow \inv{1}{|V|}$ that satisfies for all
  distinct $u,v\in V$
  \begin{equation*}
    \ell_{\precdot}(u) < \ell_{\precdot}(v) \Leftrightarrow u \precdot v.
  \end{equation*}		
\end{definition}
See \Cref{fig:working_example} for an illustrative example of $\precdot$
and $\ell_{\precdot}$.

\begin{lemma}\label{lem:order1}
  Let $G$ be a complete edge-colored graph with vertex set $V$.  Then, the
  binary relation $\precdot$ is a strict total order on~$V$ and
  $\ell_{\precdot}$ a labeling of $G$.
\end{lemma}
\begin{proof}
  Let $G$ be a complete edge-colored graph with vertex set~$V$.
  Furthermore, let $\ell^M$ be a labeling of the quotient graph
  $G[M]/\Mmax(M)$ for each $M\in \Mods[str]{G}$.  To prove that~$\precdot$
  is a strict total order on~$V$ we must show that $\precdot$ satisfies
  (R1) and (R2), \ie, $\precdot$ is trichotomous (and thus, well-defined)
  and transitive.
		
  We first verify that $\precdot$ is trichotomous.  Since there are no
  child modules of $M^{u,u}=\{u\}$ in $T_G$, we never have $u \precdot
  u$. Thus, let $u,v\in V$ be distinct. Recall that the strong modules
  $M^{u,v}$, $M^{u,v}_u$, and $M^{u,v}_v$ always exist and that they are
  unique for $u,v$. Since $\ell^{M^{u,v}}$ is a bijective map into
  $\inv{1}{|\Mmax(M^{u,v})|}$, we have either
  $\ell^{M^{u,v}}(M^{u,v}_u) < \ell^{M^{u,v}}(M^{u,v}_v)$ or
  $\ell^{M^{u,v}}(M^{u,v}_u) > \ell^{M^{u,v}}(M^{u,v}_v)$. By construction,
  we either have $u\precdot v$ or $v\precdot u$. In summary, $\precdot$ is
  trichotomous.

  We continue to show that $\precdot$ is transitive. To this end, suppose
  that $u \precdot v$ and $v \precdot w$ for some $u,v,w \in V$.  By
  trichotomy of $\precdot$, the vertices $u,v,w$ are distinct.
	
  Since $1 \leq |\{M^{u,v}, M^{v,w}, M^{u,w}\}|\leq 2$, we have to consider
  the following four cases: 1)~$M^{u,v}=M^{v,w}=M^{u,w}$,
  2)~$M^{v,w}=M^{u,w}\neq M^{u,v}$, 3)~$M^{u,v}=M^ {u,w}\neq M^{v,w}$, and
  4)~$M^{u,v}=M^{v,w}\neq M^{u,w}$.
  \begin{itemize}
  \item[1)] If $M\coloneqq M^{u,v}=M^{v,w}=M^{u,w}$, then $M_u$, $M_v$, and
    $M_w$ are distinct strong modules of $\Mmax(M)$.  Then, $u\precdot v$
    and $v\precdot w$ imply $\ell^{M}(M_u) < \ell^M(M_v) < \ell^M(M_w)$ and
    thus, we conclude $u \precdot w$.
  \item[2)] If $M^{v,w}=M^{u,w}\neq M^{u,v}$, then $M^{u,w}_u=M^{u,w}_v$.
    Since $v \precdot w$, we have
    $\ell^{M^{u,w}}(M^{u,w}_u)=\ell^{M^{u,w}}(M^{u,w}_v) <
    \ell^{M^{u,w}}(M^{u,w}_w)$ and thus, $u\precdot w$.
  \item[3)] Verified by similar arguments as in Case~2).
  \item[4)] We show that this case cannot occur.  If
    $M^{u,v}=M^{v,w}\neq M^{u,w}$, then $M^{u,v}_u=M^{u,v}_w$ and thus, we
    have $\ell^{M^{u,v}}(M^{u,v}_u)=\ell^{M^{u,v}}(M^{u,v}_w)$.  Together
    with $u \precdot v$ and $v \precdot w$, we conclude
    $\ell^{M^{u,v}}(M^{u,v}_u) < \ell^{M^{u,v}}(M^{u,v}_v) <
    \ell^{M^{u,v}}(M^{u,v}_w) = \ell^{M^{u,v}}(M^{u,v}_u)$; a
    contradiction.
  \end{itemize}
  Hence, $u\precdot v$ and $v\precdot w$ always implies $u\precdot
  w$. Therefore, $\precdot$ is transitive.  Since~$\precdot$ is
  trichotomous and transitive, the relation~$\precdot$ is a strict total
  order on~$V$.

  Since $\precdot$ is a strict total order on~$V$ we can immediately
  conclude that the map $\ell_{\precdot}$ is a labeling of $G$.
\end{proof}

We now define a second type of relation that we use to construct the
permutations for the complete edge-colored permutation graph
$(G,\ell_{\prec})$.

\begin{definition}\label{def:order2}
  Let $G$ be a complete $k$-edge-colored graph with vertex set $V$ and some
  labeling $\ell$.  For all $i\in\inv{1}{k}$, we denote with $\llcurly_i$
  the binary relation on $V$ that satisfies for all distinct $u,v\in V$
  with $\ell(u)>\ell(v)$:
  \begin{align*}
    u \llcurly_i v  & \text{ if edge } \{M^{u,v}_u,M^{u,v}_v\}
                      \text{ has color $i$ in } G[M^{u,v}]/\Mmax(M^{u,v})
                      \text{ and }\\
    v \llcurly_i u  & \text{ otherwise.}
  \end{align*}
\end{definition}

\begin{lemma}\label{lem:order2}
  Let $G$ be a complete $k$-edge-colored graph with vertex set $V$.  Suppose
  that $(G[M]/\Mmax(M),\ell^M)$ is a complete edge-colored permutation
  graph for each strong module $M\in \Mods[str]{G}$ together with some
  labeling $\ell^M$.  Then, for all $i\in\inv{1}{k}$, the binary relation
  $\llcurly_i$ defined on the labeling $\ell_{\prec}$ of $G$ is a strict
  total order on~$V$.
\end{lemma}
\begin{proof}
  In what follows, put $\ell\coloneqq \ell_{\prec}$ and let
  $i\in\inv{1}{k}$ be arbitrarily chosen.
	
  We first verify that $\llcurly_i$ is well-defined and trichotomous.
  Clearly, $G[M^{u,u}]$ consists of the single vertex $u$ only and by
  construction, we never set $u \llcurly_i u$.  Now, let $u,v\in V$ be two
  distinct vertices.  Observe that the respective strong modules
  $M^{u,v}_u,M^{u,v}_v$ and $M^{u,v}$ are unique.  This, the fact that
  $G[M^{u,v}]/\Mmax(M^{u,v})$ is complete, and each edge of
  $G[M^{u,v}]/\Mmax(M^{u,v})$ has a unique color implies that $\llcurly_i$
  is well-defined.

  Without loss of generality suppose that $\ell(u) > \ell(v)$.  The edge
  $\{M^{u,v}_u,M^{u,v}_v\}$ in $G[M^{u,v}]/\Mmax(M^{u,v})$ has a unique
  color~$j\in\inv{1}{k}$.  If $j=i$, then we have $u \llcurly_i v$.  If,
  otherwise, $j\neq i$, then we have $v \llcurly_i u$.  This together with
  the fact that $M^{u,v}_u,M^{u,v}_v$ and $M^{u,v}$ are unique for $u$ and
  $v$ implies that $\llcurly_i$ is trichotomous.
		
  We continue to show that $\llcurly_i$ is transitive.  Suppose that
  $u \llcurly_i v$ and $v \llcurly_i w$ for some $u,v,w \in V$.  By
  trichotomy of~$\llcurly_i$, the vertices $u,v,w$ are distinct.

  Since $1\leq |\{M^{u,v}, M^{v,w}, M^{u,w}\}|\leq 2$, we have to consider
  the following four cases: 1)~$M^{u,v}=M^{v,w}=M^{u,w}$,
  2)~$M^{v,w}=M^{u,w}\neq M^{u,v}$, 3)~$M^{u,v}=M^{u,w}\neq M^{v,w}$, and
  4)~$M^{u,v}=M^{v,w}\neq M^{u,w}$.

  \begin{itemize}
  \item[1)] If $M\coloneqq M^{u,v}=M^{v,w}=M^{u,w}$, then $M_u$, $M_v$, and
    $M_w$ are distinct strong modules of $\Mmax(M)$.  Hence, the quotient
    graph $G[M]/\Mmax(M)$ contains the vertices $M_u$, $M_v$, and $M_w$.
    Since $G[M]/\Mmax(M)$ is complete (by~\Cref{rem:subgraph}), the edges
    $\{M_u,M_v\}$, $\{M_v,M_w\}$, and $\{M_u,M_w\}$ exist in
    $G[M]/\Mmax(M)$ and have each a unique color.  For simplicity, we
    denote the color of edge $\{M_u,M_v\}$, $\{M_v,M_w\}$, and
    $\{M_u,M_w\}$ by $c_{uv}$, $c_{vw}$, and $c_{uw}$, respectively.  By
    assumption $G[M]/\Mmax(M)$ is a complete edge-colored permutation
    graph. Thus, we can apply \Cref{lem:rainbow_nonprime2} to conclude that
    $G[M]/\Mmax(M)$ cannot contain a rainbow triangle, \ie,
    $1\leq |\{c_{uv},c_{vw},c_{uw}\}|\leq 2$.  In what follows, we refer to
    the latter condition as Condition ``\emph{2C}''.  We now consider six
    different cases with respect to the total order of $\ell(u)$,
    $\ell(v)$, and $\ell(w)$.
    \begin{itemize}
    \item[\em Case $\ell(u)<\ell(v)<\ell(w)$:] In this case,
      $u \llcurly_i v$ implies $c_{uv}\in\inv{1}{k}\setminus\{i\}$ and
      $v \llcurly_i w$ implies $c_{vw}\in\inv{1}{k}\setminus\{i\}$.
      Suppose first that $c_{uv}\neq c_{vw}$. By~\emph{2C}, we have
      $c_{uw} \in \{c_{uv},c_{vw}\}$.  Hence, $c_{uw}\neq i$ which together
      with $\ell(u)<\ell(w)$ implies $u\llcurly_i w$.
				
      Now consider $c\coloneqq
      c_{uv}=c_{vw}$. By~\Cref{lem:forbidden_subgraph}, the case
      $c_{uw}\neq c$ cannot occur. Consequently, we have $c_{uw}=c\neq i$
      which together with $\ell(u)<\ell(w)$ implies $u \llcurly_i w$.

    \item[\em Case $\ell(w)<\ell(v)<\ell(u)$:] In this case,
      $u \llcurly_i v$ and $v \llcurly_i w$ imply $c_{uv}=c_{vw}=i$.  By
      \Cref{lem:forbidden_subgraph}, $c_{uw}=i$. This together with
      $\ell(w)<\ell(u)$ implies $u \llcurly_i w$.

    \item[\em Case $\ell(u)<\ell(w)<\ell(v)$:] In this case,
      $u \llcurly_i v$ implies $c_{uv}\in\inv{1}{k}\setminus\{i\}$ and
      $v \llcurly_i w$ implies $c_{vw}=i$.  By~\emph{2C}, we have
      $c_{uw} \in \{c_{uv},i\}$. However, by~\Cref{lem:forbidden_subgraph},
      we must have $c_{uw}\neq i$ which together with $\ell(u)<\ell(w)$
      implies $u \llcurly_i w$.
			
    \item[\em Case $\ell(v)<\ell(u)<\ell(w)$:] In this case,
      $u \llcurly_i v$ implies $c_{uv}=i$ and $v \llcurly_i w$ implies
      $c_{vw}\in\inv{1}{k}\setminus\{i\}$.  By~\emph{2C}, we have
      $c_{uw} \in \{c_{vw},i\}$.  Again, by~\Cref{lem:forbidden_subgraph},
      $c_{uw}\neq i$ which together with $\ell(u)<\ell(w)$ implies
      $u \llcurly_i w$.

    \item[\em Case $\ell(v)<\ell(w)<\ell(u)$:] In this case,
      $u \llcurly_i v$ implies $c_{uv}=i$ and $v \llcurly_i w$ implies
      $c_{vw}\in\inv{1}{k}\setminus\{i\}$.  By~\emph{2C}, we have
      $c_{uw} \in \{c_{vw},i\}$.  However,
      by~\Cref{lem:forbidden_subgraph}, it must hold $c_{uw}\neq c_{vw}$
      and thus, $c_{uw}=i$. This together with $\ell(w)<\ell(u)$ implies
      $u \llcurly_i w$.
				
    \item[\em Case $\ell(w)<\ell(u)<\ell(v)$:] In this case,
      $u \llcurly_i v$ implies $c_{uv}\in\inv{1}{k}\setminus\{i\}$ and
      $v \llcurly_i w$ implies $c_{vw}=i$.  By~\emph{2C}, we have
      $c_{uw} \in \{c_{uv},i\}$.  Again, by ~\Cref{lem:forbidden_subgraph},
      we have $c_{uw}\neq c_{uv}$ and thus, $c_{uw}=i$.  This together with
      $\ell(w)<\ell(u)$ implies $u \llcurly_i w$.
    \end{itemize}
    In summary, $u \llcurly_i v$ and $v \llcurly_i w$ implies
    $u \llcurly_i w$ in each of the six cases.
	 			
  \item[2)] If $M\coloneqq M^{v,w}=M^{u,w}\neq M^{u,v}$, then
    $M_u=M_v\neq M_w$.  This, in particular, implies
    $e\coloneqq\{M_v,M_w\}=\{M_u,M_w\}\in E(G[M]/\Mmax(M))$.  Since
    $\ell^{M}$ is a bijective map into $\inv{1}{|\Mmax(M)|}$, we have
    either $\ell^{M}(M_u)=\ell^{M}(M_v)>\ell^{M}(M_w)$ or
    $\ell^{M}(M_u)=\ell^{M}(M_v)<\ell^{M}(M_w)$.
			
    If $\ell^{M}(M_u)=\ell^{M}(M_v)>\ell^{M}(M_w)$, then, by construction,
    we have $\ell(u)>\ell(w)$ and $\ell(v)>\ell(w)$.  Since
    $\ell(v)>\ell(w)$ and $v \llcurly_i w$, the edge~$e$ has color~$i$.
    This together with $\ell(u)>\ell(w)$ implies $u \llcurly_i w$.
			
    If $\ell^{M}(M_u)=\ell^{M}(M_v)<\ell^{M}(M_w)$, then, by construction,
    we have $\ell(u)<\ell(w)$ and $\ell(v)<\ell(w)$.  Since
    $\ell(v)<\ell(w)$ and $v \llcurly_i w$, the edge~$e$ has
    color~$j\in \inv{1}{k}\setminus\{i\}$.  This together with
    $\ell(u)<\ell(w)$ implies $u \llcurly_i w$.

  \item[3)] Verified by similar arguments as in Case~2).
			 	
  \item[4)] We show that the case $M^{u,v}=M^{v,w}\neq M^{u,w}$ cannot
    occur.  If $M\coloneqq M^{u,v}=M^{v,w}\neq M^{u,w}$, then
    $M_u=M_w\neq M_v$.  This, in particular, implies
    $e\coloneqq\{M_u,M_v\}=\{M_w,M_v\}\in E(G[M]/\Mmax(M))$.
    Since~$\ell^{M}$ is a bijective map into $\inv{1}{|\Mmax(M)|}$, we have
    either $\ell^{M}(M_u)=\ell^{M}(M_w)>\ell^{M}(M_v)$ or
    $\ell^{M}(M_u)=\ell^{M}(M_w)<\ell^{M}(M_v)$.
			
    If $\ell^{M}(M_u)=\ell^{M}(M_w)>\ell^{M}(M_v)$, then, by construction,
    we have $\ell(u)>\ell(v)$ and $\ell(w)>\ell(v)$.  Since
    $\ell(u)>\ell(v)$ and $u \llcurly_i v$, the edge $e$ must have color
    $i$.  However, $\ell(w)>\ell(v)$ and $v \llcurly_i w$ implies that the
    edge~$e$ has color~$j\in\inv{1}{k}\setminus\{i\}$; a contradiction.
			
    If $\ell^{M}(M_u)=\ell^{M}(M_w)<\ell^{M}(M_v)$, then, by construction,
    we have $\ell(u)<\ell(v)$ and $\ell(w)<\ell(v)$.  Since
    $\ell(u)<\ell(v)$ and $u \llcurly_i v$, the edge~$e$ must have
    color~$j\in\inv{1}{k}\setminus\{i\}$.  However, $\ell(w)<\ell(v)$ and
    $v \llcurly_i w$ implies that the edge $e$ must have color $i$; a
    contradiction.
  \end{itemize}
  In each case, $u \llcurly_i v$ and $v \llcurly_i w$ implies
  $u \llcurly_i w$, \ie, $\llcurly_i$ is transitive. This, together with
  the trichotomy of $\llcurly_i$ implies that $\llcurly_i$ is a strict
  total order on~$V$.  Since $i\in\inv{1}{k}$ was chosen arbitrarily, all
  $\llcurly_1,\dots, \llcurly_k$ are strict total orders on $V$.
\end{proof}

We are now in the position to show that the converse of
\Cref{lem:colored_perm_iff2} is valid as well.

\begin{proposition}\label{lem:colored_perm_iff2-b}
  Let $G$ be a complete edge-colored graph such that the quotient graph
  $G[M]/\Mmax(M)$ of each strong module $M\in \Mods[str]{G}$ is a complete
  edge-colored permutation graph.  Then $G$ is a complete edge-colored
  permutation graph.
\end{proposition}
\begin{proof}
  Let $G=(V,E_1,\dots,E_k)$ be a complete $k$-edge-colored graph.  By assumption
  there exists for each strong module $M\in \Mods[str]{G}$ a labeling
  $\ell^M$ of the quotient graph $G[M]/\Mmax(M)$ such that
  $(G[M]/\Mmax(M),\ell^M)$ is a complete edge-colored permutation graph.
		
  Now, let $\precdot$ and $\ell_{\precdot}$ be as specified in
  \Cref{def:order1}.  By \Cref{lem:order1} the map
  $\ell\coloneqq \ell_{\precdot}$ is a labeling of $G$.  Moreover, by Lemma~\ref{lem:order2} we have for all $i\in\inv{1}{k}$ that the binary
  relation $\llcurly_i$ defined on the labeling $\ell$ of $G$ is a strict
  total order on~$V$.  This immediately implies that the map
  $\pi_i\dis \inv{1}{|V|} \rightarrow \inv{1}{|V|}$ that satisfies for all
  distinct $u,v\in V$
\begin{equation*}
  \pi^{-1}_i(\ell(u)) < \pi^{-1}_i(\ell(v)) \Leftrightarrow u \llcurly_i v
\end{equation*}
is a permutation of length~$|V|$ for all $i\in \inv{1}{k}$.
		
It remains show that $(G,\ell)$ is a complete edge-colored permutation
graph of $\pi_1,\dots,\pi_k$.  For this purpose, let $u,v$ be two distinct
vertices of $G=(V,E_1,\dots,E_k)$ and $i\in\inv{1}{k}$.  Without loss of
generality let $\ell(u)>\ell(v)$.  If $\{u,v\}\in E_i$, then, by
definition, $\{M^{u,v}_u,M^{u,v}_v\}$ must have color~$i$ in
$G[M^{u,v}]/\Mmax(M^{u,v})$.  Therefore, we have $u \llcurly_i v$ which
implies $\pi^{-1}_i(\ell(u))<\pi^{-1}_i(\ell(v))$.  By analogous reasoning,
$\pi^{-1}_i(\ell(u))<\pi^{-1}_i(\ell(v))$ and $\ell(u)>\ell(v)$ implies
$\{u,v\}\in E_i$.  In summary, we showed for all distinct $u,v\in V$ and
all $i\in \inv{1}{k}$ that
\begin{equation*}
  \{u,v\} \in E_i \Leftrightarrow \ell(u)>\ell(v) \text{ and } \pi^{-1}_i(\ell(u))<\pi^{-1}_i(\ell(v)).
\end{equation*}
Consequently, $(G,\ell)$ is a complete edge-colored permutation graph of
$\pi_1,\dots,\pi_k$.
\end{proof}

\subsection{A Characterization in Terms of Prime
  Modules}\label{section:item6}

Intriguingly, we can strengthen the results of Section~\ref{sec:strongM} by
considering strong prime modules only. We show here that complete
edge-colored permutation graphs are completely characterized by the fact
that the quotient graph of each strong prime module is a complete
edge-colored permutation graph.
Hence, we show the equivalence of \Cref{thm:characterization_all_k_PG}~(i) and \Cref{thm:characterization_all_k_PG}~(iv).

\begin{proposition}\label{thm:k_colored_permgraph_iff}
  Let $G$ be a complete edge-colored graph. Assume that
  the quotient graph of each strong prime module in the modular
  decomposition of $G$ is a complete edge-colored permutation graph. Then
  $G$ is a complete edge-colored permutation graph.
\end{proposition}
\begin{proof}
  First note that the quotient graph $G[M]/\Mmax(M)$ of a strong module $M$
  of $G$ is always complete (\Cref{rem:subgraph}).  The quotient graph
  $G[M]/\Mmax(M)$ of a series module~$M$ of~$G$ is complete and
  $1$-edge-colored.  In this case, $G[M]/\Mmax(M)$ is the complete $1$-edge-colored
  permutation graph of the permutation $\ol{\iota}$.  By assumption, the
  quotient graph of each strong prime module of $G$ is a complete
  edge-colored permutation graph.  In summary, the quotient graph
  $G[M]/\Mmax(M)$ of \emph{every} strong module $M\in \Mods[str]{G}$ is a
  complete edge-colored permutation graph. Hence, we can apply
  \Cref{lem:colored_perm_iff2-b} to conclude that $G$ is a complete
  edge-colored permutation graph.
\end{proof}

\begin{corollary}\label{thm:k_colored_permgraph_iff-b}
  Let $G$ be a complete edge-colored permutation graph.  Then the quotient
  graph of each strong prime module in the modular decomposition of $G$ is
  a complete edge-colored permutation graph.  In particular, for every
  strong prime module on at least three vertices the quotient graph must be
  $2$-edge-colored.
\end{corollary}
\begin{proof}
  Let $G$ be a complete edge-colored permutation graph.  By
  \Cref{lem:colored_perm_iff2}, the quotient graph of each strong module of~$G$
  and thus, of each strong prime module of $G$, is a complete edge-colored
  permutation graph.

  We continue to show that the quotient graph of each strong prime module
  on at least three vertices must be 2-edge-colored.  Assume, for contradiction,
  that $G$ contains a strong prime module~$M$ with $|M|\geq 3$ whose
  quotient graph $G[M]/\Mmax(M)$ is not a $2$-edge-colored graph.  By
  \Cref{rem:subgraph}, $G[M]/\Mmax(M)$ is complete. However, this implies
  that $G[M]/\Mmax(M)$ must be $k$-edge-colored for some $k\geq 3$, as
  otherwise, it would be a simple permutation graph for $\ol\iota$.

  Since $G[M]/\Mmax(M)$ is primitive (cf.\
  \Cref{lem:PrimeImpliesPrimitive}), we can apply \Cref{lem:prime_rainbow}
  to conclude that $G[M]/\Mmax(M)$ contains a rainbow triangle. It is easy
	to verify that this implies that $G$
  must contain a rainbow triangle; a contradiction to
  \Cref{lem:rainbow_nonprime2} and the assumption that $G$ is a complete
  edge-colored permutation graph.
\end{proof}

\subsection{A Characterization via Monochromatic Subgraphs}\label{sec:mono}
In this section, we show that a graph is a complete edge-colored
permutation graph if and only if it does not contain a rainbow triangle and
each monochromatic subgraph of it is a simple permutation graph.  
Hence, we show the equivalence of \Cref{thm:characterization_all_k_PG}~(i) and \Cref{thm:characterization_all_k_PG}~(v).
To this end, we first need the following result:

\begin{lemma}\label{lem:characterization-2nd-case} 
  Let $G$ be a complete edge-colored graph. Assume that $G_{|i}$ (together
  with some labeling $\ell_i$) is a simple permutation graph for all
  $i\in\inv{1}{k}$ and the quotient graph of each strong prime module of
  $G$ on at least three vertices is a 2-edge-colored graph.  Then $G$ is a
  complete edge-colored permutation graph.
\end{lemma}
\begin{proof}
  Let $G=(V,E_1,\dots,E_k)$ be a complete $k$-edge-colored graph.
  We proceed to show that all conditions of \Cref{lem:colored_perm_iff2-b}
  are satisfied in order to show that $G$ is a complete edge-colored
  permutation graph. To this end, let us examine the quotient graph of a
  strong module $M$ of $G$.

  If $M$ is series, then its quotient graph $G[M]/\Mmax(M)$ is a complete
  1-edge-colored graph and thus, a complete 1-edge-colored permutation graph
  for~$\ol{\iota}$.

  If $M$ is prime, then, by \Cref{rem:subgraph}, $G[M]/\Mmax(M)$ is
  complete.  If $|M|\leq 2$, then $|M|=1$ as otherwise, $M$ would be
  series.  Hence, $G[M]/\Mmax(M)$ is the single vertex graph and thus, a
  complete edge-colored permutation graph.  Now suppose that $|M|\geq 3$. By
  assumption, $G[M]/\Mmax(M)$ is $2$-edge-colored.  Hence, for the two distinct
  colors~$i$ and~$j$ used to color $G[M]/\Mmax(M)$, there must be edges
  with color~$i$ and~$j$, respectively.  By assumption $G_{|i}$ 
	and $G_{|j}$ are a simple
  permutation graphs.  By~\Cref{lem:simplePerm-hereditary},
  $G_{|i}[N]$ and $G_{|j}[N]$ are simple permutation graphs for all
  $N\subseteq V$.  Let $N\subseteq V$ be chosen such that~$N$ contains
  exactly one vertex from each child module of~$M$.  By
  \Cref{rem:subgraph}, $G[M]/\Mmax(M)$ is isomorphic to the induced
  subgraph $G[N]$ of $G$.  In particular, the \linebreak $h$-th
  monochromatic subgraph $G[M]/\Mmax(M)_{|h}$ of $G[M]/\Mmax(M)$ must be
  isomorphic to $G_{|h}[N]$, $h\in\{i,j\}$.  As argued above, $G_{|i}[N]$ and $G_{|j}[N]$
  are simple permutation graphs and therefore, $G[M]/\Mmax(M)_{|i}$ and
  $G[M]/\Mmax(M)_{|j}$ are both simple permutation graphs.  Since
  $G[M]/\Mmax(M)$ is 2-edge-colored, we can
  apply~\Cref{cor:two_colored_one_label} to conclude that $G[M]/\Mmax(M)$
  is a complete edge-colored permutation graph.

  For each strong module $M$ of $G$, we showed that the quotient graph
  $G[M]/\Mmax(M)$ is a complete edge-colored permutation graph.  Then, the
  claim follows by \Cref{lem:colored_perm_iff2-b}.
\end{proof}

The following two propositions prove the equivalence of Items~(i) and~(v)
of \Cref{thm:characterization_all_k_PG}.
\begin{proposition}\label{pro:characterization}
  Let $(G,\ell)$ be a complete edge-colored permutation graph. Then $G$
  does not contain a rainbow triangle and each monochromatic subgraph of
  $G$ is together with $\ell$ a simple permutation graph.
\end{proposition}
\begin{proof}
  Let $(G,\ell)$ be a complete $k$-edge-colored permutation graph for the
  permutations $\pi_1,\dots,\pi_k \in \Mc{P}_{|V|}$.
  By~\Cref{lem:rainbow_nonprime2}, $G$ cannot contain a rainbow triangle.
  By definition, $(G_{|i},\ell)$ is be a simple permutation graph of~$\pi_i$.
\end{proof}

\begin{proposition}\label{pro:characterization-unlabeled} 
  Let $G$ be a complete edge-colored graph.  If every monochromatic
  subgraph of~$G$ is a simple permutation graph and~$G$ does not contain a
  rainbow triangle, then~$G$ is a complete edge-colored permutation graph.
\end{proposition}
\begin{proof} 
  The statement is trivially true if $G\simeq K_1$.
  Thus, let $G = (V,E_1,\dots, E_k)$ be a complete $k$-edge-colored graph with
  $|V|\geq 2$.

  By contraposition, we show that if $G$ is not a complete edge-colored
  permutation graph, then there exists a monochromatic subgraph of $G$ that
  is not a simple permutation graph or $G$ contains a rainbow triangle.

  Assume that $G$ is not a complete edge-colored permutation graph.  Since
  $G$ contains at least two vertices, we have that $k\geq 1$.  If $k=1$,
  then $G$ would be a complete $1$-edge-colored permutation graph for
  $\pi_1 = \ol{\iota}$.  Thus, we assume that $k\geq 2$.  Since $G$ is not
  a complete edge-colored permutation graph, for all labelings
  $\ell\dis V \rightarrow \inv{1}{|V|}$ of $G$ there are no permutations
  $\pi_1,\dots,\pi_k\in \Mc{P}_{|V|}$ such that $(G,\ell)$ is a complete
  $k$-edge-colored permutation graph of $\pi_1,\dots,\pi_k$.
  By definition, for all labelings $\ell$, $G$ contains always a
  monochromatic subgraph $G_{|i}=(V,E_i)$ such that $(G_{|i},\ell)$ is not
  a simple permutation graph of $\pi_i$.  Note, this does not imply that
  $G_{|i}$ is not a simple permutation graph at all, since we might find a
  different labeling $\ell_i$ of $G_{|i}$ such that $(G_{|i},\ell_i)$ is a
  simple permutation graph for some $\pi'_i$.

  If, however, for all possible labelings $\ell_i$, the labeled graph
  $(G_{|i},\ell_i)$ is not a simple permutation graph, then we are done,
  since we found a monochromatic subgraph of $G$ that is not a simple
  permutation graph at all.

  Thus, we are left with the situation that $G$ is not a complete
  edge-colored permutation graph and for all $i\in \inv{1}{k}$ there is
  some labeling $\ell_i$ such that $(G_{|i},\ell_i)$ is a simple
  permutation graph. Note, that this implies that $k>2$ as otherwise,
  $E(G) = E(G_{|1})\cup E(\overline{G_{|1}})$ and, by
  \Cref{cor:two_colored_one_label}, $G$ would be a complete $2$-edge-colored
  permutation graph.  It remains to show that $G$ contains a rainbow
  triangle.

  We consider here the two cases: either Case~(i) for none of the strong
  prime modules $M$ in $G$ the quotient graph $G[M]/\Mmax(M)$ has at least
  three colors or Case~(ii) there exists a strong prime module $M$ in $G$
  such that $G[M]/\Mmax(M)$ has at least three colors.

  Let us first consider Case (i). Then, for every strong prime module $M$
  of $G$ on at least three vertices the quotient graph $G[M]/\Mmax(M)$ is
  $2$-edge-colored (as otherwise $M$ would be series).  By assumption, for all
  $i\in \inv{1}{k}$ the monochromatic subgraph $G_{|i}$ is a simple
  permutation graph.  Hence, we can apply
  \Cref{lem:characterization-2nd-case} to conclude that $G$ is a complete
  edge-colored permutation graph; a contradiction.  Therefore, Case~(i)
  cannot occur and we are left with Case~(ii).

  For Case (ii), let $M$ be a strong prime module such that $G[M]/\Mmax(M)$
  has at least three colors.  Note that this implies that $|M|\neq 1$.  By
  \Cref{lem:PrimeImpliesPrimitive}, $G[M]/\Mmax(M)$ is primitive. By
  \Cref{lem:prime_rainbow}, $G[M]/\Mmax(M)$ contains a rainbow triangle
  induced by some vertices $M',M'',M'''\in \Mmax(M)$.  By definition, there are vertices
  $x\in M'$, $y\in M''$, and $z\in M'''$ that induce the rainbow triangle $\tri_{xyz}$ in
  $G[M]$ and thus, in $G$.

  In summary, we have shown that if the complete edge-colored graph~$G$ is
  not a complete edge-colored permutation graph, then there exists a
  monochromatic subgraph of $G$ that is not a simple permutation graph or
  $G$ contains a rainbow triangle, which completes the proof.
\end{proof}

\subsection{Proof of \Cref{thm:characterization_all_k_PG}}
\label{section:mainproof}

Theorem~\ref{thm:characterization_all_k_PG} is a direct consequence of
  the results described in this section:

\begin{proof}[Proof of \Cref{thm:characterization_all_k_PG}]
  \Cref{cor:i_iii-b} and \Cref{pro:i_iii} provide the equivalence between Items~(i) and~(ii).
  \Cref{lem:colored_perm_iff2} and \Cref{lem:colored_perm_iff2-b} provide
  the equivalence between Items~(i) and~(iii).  Moreover, by
  \Cref{thm:k_colored_permgraph_iff} and
  \Cref{thm:k_colored_permgraph_iff-b}, Items~(i) and~(iv) are equivalent.
  Finally, by \Cref{pro:characterization} and
  \Cref{pro:characterization-unlabeled}, Items~(i) and~(v) are equivalent.
\end{proof}

\section{Recognition of Complete Edge-Colored Permutation Graphs}
\label{sec:recognition}

In this section we show that complete edge-colored permutation graphs can be
recognized in $\Mc{O}(|V|^2)$-time. Moreover, we also show that (in the 
affirmative case) the labeling~$\ell$ and the
permutations~$\pi_1,\dots,\pi_k$ such that $(G,\ell)$ is a complete
$k$-edge-colored permutation graph of $\pi_1,\dots,\pi_k$ can be constructed in
polynomial time.

Simple permutation graphs $G=(V,E)$ can be recognized in $\Mc{O}(|V|+|E|)$-time \cite{Crespelle2010,Mcconnell1999}. These algorithms also construct a
corresponding permutation~$\pi$ and labeling~$\ell$ of $G$.

Let us now consider a complete $k$-edge-colored graph~$G$.  As the number of
colors~$k$ is always bounded by the number of edges of~$G$, we immediately
obtain a polynomial-time algorithm to recognize complete edge-colored
permutation graphs (cf.\ \Cref{thm:characterization_all_k_PG}~(v)):
\begin{enumerate}
\item[1)] For all $i\in\inv{1}{k}$, we check whether the monochromatic
  subgraph~$G_{|i}$ is a simple permutation graph using the algorithms
  described in \cite{Crespelle2010,Mcconnell1999}.
\item[2)] In addition, we verify that $G$ does not contain a rainbow
  triangle.
\end{enumerate}
Such an algorithm allows us to verify that $G$ is a complete edge-colored
permutation graph or not. However, we are also interested in computing a
labeling~$\ell$ of $G$ and the corresponding underlying
permutations~$\pi_1,\dots,\pi_k$ such that $(G,\ell)$ is a complete
edge-colored permutation graph of $\pi_1,\dots,\pi_k$.

The proofs in \Cref{sec:strongM}, however, are  constructive and thus
provide a recognition algorithm that explicitly computes  the required
labeling~$\ell$ as well as  the underlying permutations of a complete
edge-colored permutation graph. More precisely, we may use the
  following procedure:
\begin{enumerate}
\item[1)] Compute the modular decomposition of~$G$ and, thus, obtain all
					strong modules of~$G$.  If there is a strong prime
  module~$M$ of~$G$ with $|M|\geq 3$ whose quotient graph $G[M]/\Mmax(M)$
  contains more than $2$ colors, then $G$ cannot be a complete edge-colored
  permutation graph (cf.\ \Cref{thm:characterization_all_k_PG}~(iv)).  In
  this case, the algorithm stops and returns \emph{false}.  Otherwise, each
  of these quotient graphs on at least three vertices is $2$-edge-colored and we
  can proceed with Step (2).
\item[2)] For all strong prime modules~$M$ on at least three vertices,
  check whether one of the two monochromatic subgraphs of $G[M]/\Mmax(M)$
  is a simple permutation graph and, thus, obtain a labeling~$\ell^M$ of
  the quotient graph $G[M]/\Mmax(M)$. 
\item[3)] For all remaining strong modules~$M$ (those that are either
  prime modules with $|M|=1$ or series modules) choose an arbitrarily
  labeling~$\ell^M$ of the quotient graph $G[M]/\Mmax(M)$.
\item[4)] Construct $\precdot$ and $\ell_{\precdot}$ as specified in
  \Cref{def:order1}.	
\item[5)] For all $i\in \inv{1}{k}$, construct $\llcurly_i$ as specified in
  \Cref{def:order2} and permutation~$\pi_i$ as defined in the proof of
  \Cref{lem:colored_perm_iff2-b}.
\end{enumerate}
This algorithm correctly determines whether $G$ is a complete edge-colored
permutation graph or not (it essentially checks Item~(iv) in
\Cref{thm:characterization_all_k_PG}) and correctly determines the labeling
$\ell_{\precdot}$ of $G$ (cf.\ \Cref{lem:order1}) and the corresponding
permutations $\pi_1,\dots,\pi_k$ (cf.\ proof of
\Cref{lem:colored_perm_iff2-b}).

%

\begin{theorem}\label{thm:alg}
  Let $G=(V,E)$ be a complete $k$-edge-colored graph.  Then, it can be verified in
  $\Mc{O}(|V|^2)$-time whether $G$ is a complete edge-colored permutation graph
  or not.  In the affirmative case, a labeling~$\ell$ and
  permutations~$\pi_1,\dots,\pi_k$ such that $(G,\ell)$ is a complete
  edge-colored permutation graph for $\pi_1,\dots,\pi_k$ can be constructed
  in $\Mc{O}(|V|^2)$-time.
\end{theorem}
\begin{proof}
  The correctness of the algorithm follows from the discussion above.
  Thus, let us examine its running time.

  In Step (1) we have to compute the modular decomposition tree~$T$ of~$G$,
  which can be done in $\Mc{O}(|V|^2)$-time~\cite{Hellmuth:16a,EHMS:94}.  While
  computing~$T$ we can also store the information which module is
  associated to which vertex in~$T$ within the same time complexity.
  At the same time, we can also store the information which
  inclusion-minimal strong module $M^{x,y}$ of $G$ contains $x$ and $y$ as
  well as the child modules $M_x^{x,y}$ and $M_y^{x,y}$ that contain
  $x$ and $y$, respectively, for all $x,y\in V$ in $\Mc{O}(|V|^2)$-time.
 
  We then have to compute the quotient graph $G[M]/\Mmax(M)$ for all strong
  modules $M$. The quotient graph of an individual strong module $M$ can be
  computed as follows: The module $M$ is associated with a vertex $u$ in~$T$ 
  that has children $u_1,\dots,u_r$, $r\geq 2$. Each child $u_i$ is
  associated with a strong module $M_i$ and hence,
  $\Mmax(M) = \{M_1,\dots,M_r\}$.  We now take from each module $M_i$ one
  vertex $v_i$ of $G$, which can be done in $r$ constant time steps. We
  finally determine the colored edges $\{v_i,v_j\}$, $1\leq i<j\leq r$
  which can be done in $r(r-1)/2$ constant time steps.  Note,
  $r = \deg_T(u)-1$ and thus, the quotient graph of a strong
  module $M$ associated with $u$ in $T$ can be constructed in
  $\deg_T(u)-1 + (\deg_T(u)-1)(\deg_T(u)-2)/2$ steps, each of which
    requires only a constant number of operations.

  To compute the quotient graphs of all modules we simply process each vertex in
  $T$ and its children, and construct the quotient graphs as outlined above and
  obtain
  \begin{equation} \label{eq:sum}
    \begin{split}
      \sum_{u\in V(T)} & (\deg_T(u)-1 + (\deg_T(u)-1)(\deg_T(u)-2)/2)
       \leq \sum_{u\in V(T)}  \deg_T(u) + \deg_T(u)^2 \\
       &\leq \sum_{u\in V(T)} \deg_T(u) + |L(T)| \sum_{u\in V(T)} \deg_T(u) 
       =(1+|L(T)|) \sum_{u\in V(T)} \deg_T(u)\\
      & = 2 (1+|L(T)|) |E(T)|
      = 2 (1+|L(T)|) (|V(T)|-1) \in \Mc{O}(|V|^2),\\
    \end{split} 
  \end{equation}	
  where $L(T)$ denotes the leaves of~$T$.
  In the last step we make use of~\cite[Lemma 1]{hellmuth2015phylogenomics}
  that states that $V(T)\in \Mc{O}(|L(T)|)$ and the fact that $L(T)=V$.

  We can now proceed to examine the time complexity of Step~(2).  For each
  strong prime module~$M$ (associated with vertex $u$ in $T$) on at least three vertices we need to check 
  whether one of the two monochromatic subgraphs $H$ of $G[M]/\Mmax(M)$ is
  a simple permutation graph. Let $r=\deg(u)-1$. Hence, $|V(H)|=r$ and
  $|E(H)|\leq r(r-1)/2$.  Now, we can apply an $\Mc{O}(|V(H)| + |E(H)|)$-time
  algorithm as provided \eg\ in ~\cite{Crespelle2010,Mcconnell1999} to
  check if $H$ is a simple permutation graph or not. In the affirmative
  case, we also obtain a labeling~$\ell^M$ of $H$ and, thus, a labeling of $G[M]/\Mmax(M)$.  
  By analogous arguments as in \Cref{eq:sum}, we can observe that applying these algorithms for each strong prime 
  module~$M$ with $|M|\geq 3$ can be done in $\Mc{O}(|V|^2)$-time.
		
  In Step (3), we choose for all remaining strong modules~$M$ (associated
  with some vertex $u$ in $T$) an arbitrarily labeling~$\ell^M$ of the quotient
  graph $G[M]/\Mmax(M)$, which can be done in $\deg(u)-1$ constant time
  steps. Again making use of $V(T)\in \Mc{O}(|L(T)|)$ and $L(T)=V$ we
    obtain
  \begin{equation*}
    \sum_{u\in V(T)} \deg_T(u)-1 \leq \sum_{u\in V(T)} \deg_T(u) = 2|E(T)| =
    2(|V(T)|-1)\in \Mc{O}(|V|).
  \end{equation*}
  
  In Step (4) we use the precomputed $M^{x,y}$ to construct
  $\precdot$ and $\ell_{\precdot}$ as specified in \Cref{def:order1}. This
  task takes $\Mc{O}(|V|^2)$-time.

  Finally, in Step (5), we construct $\llcurly_i$ as specified in
  \Cref{def:order2} for all $i\in \inv{1}{k}$. This can be done in
  $\Mc{O}(|V|^2)$-time, since for each pair $\{x,y\}$ with $x,y\in V$ we can
  check in constant time whether the edge $\{M^{x,y}_x,M^{x,y}_y\}$ has
  color~$i$ in $G[M^{x,y}]/\Mmax(M^{x,y})$ (since $M^{x,y}_x$, $M^{x,y}_y$, and
  $G[M^{x,y}]/\Mmax(M^{x,y})$ have been precomputed).  The
  permutation~$\pi_i$ as defined in the proof of
  \Cref{lem:colored_perm_iff2-b} simply refers to a re-ordering of the
  vertices in $V$. Therefore, it can be performed in $\Mc{O}(|V|^2)$-time.

  Hence, the overall running time is $\Mc{O}(|V|^2)$.
\end{proof}

\section{Connections between Complete Edge-Colored Permutation Graphs and other Graph Classes}
\label{sec:graph_classes}

In this section, we show the relationship between complete edge-colored
permutation graphs and other graph classes.  The first simple observation
is that complete $k$-edge-colored permutation graphs do not contain rainbow
triangles (cf.\ \Cref{lem:rainbow_nonprime2}).  This is precisely the
definition of Gallai colorings
\cite{Gyarfas2004,Koerner1992,Gyarfas2010}. We therefore have
\begin{corollary}
  \label{coro:perm_has_gallai}
  The coloring of a complete $k$-edge-colored permutation graph is a Gallai
  coloring with $k$ colors.
\end{corollary}

We continue by showing that the class of complete edge-colored permutation graphs is
a generalization of so-called symbolic ultrametrics \cite{Bocker1998},
which have important applications in computational biology~\cite{hellmuth2013orthology,hellmuth2015phylogenomics,hellmuth2016sequence}.
They are defined as follows:
\begin{definition}[Symbolic Ultrametric]
  Let $X$ be a non-empty finite set, let $M=\inv{1}{k}$ for some
  $k\in\mathbb{N}$. A surjective map
  $\delta\colon X \times X \rightarrow M$ is a \emph{symbolic
    ultrametric} if it satisfies the following conditions:
  \begin{align*}
    &\text{(U1)}\quad \delta(x,y)=\delta(y,x) \text{ for all } x,y \in X\text{; }\\[1ex]
    &\text{(U2)}\quad |\{\delta(x,y),\delta(x,z),\delta(y,z)\}| \leq 2 \text{ for all } x,y,z \in X\text{; and}\\[1ex]
    &\text{(U3)}\quad \text{there exists no subset } \{x,y,u,v\} \in \tbinom{X}{4} \text{ such that }\\
    &\hphantom{\text{(U3)}}\quad  \delta(x,y)=\delta(y,u)=\delta(u,v)\neq\delta(v,y)=\delta(x,v)=\delta(x,u)\text{.}
  \end{align*}
\end{definition}
The set $M$ is often used to represent so-called evolutionary events~\cite{hellmuth2013orthology} and can be replaced by any finite set of
symbols or the set $\inv{1}{k}$ as done here. A complete edge-colored graph
$G_{\delta}$ with vertex set~$X$ can readily be obtained from~$\delta$, by
putting the color $\delta(x,y)=\delta(y,x)$ on the edge $\{x,y\}$ for all
distinct $x,y\in X$.  Since $\delta$ is surjective, we obtain the complete
$k$-edge-colored graph $G_{\delta}=(X,E_1,\dots,E_{k})$ as the \emph{graph
  representation of $\delta$}.

Several characterizations of symbolic ultrametrics are known~\cite{Bocker1998,hellmuth2013orthology}, including one that relates them to cographs, \ie, graphs that do not contain induced paths on four
  vertices~\cite{CORNEIL1981163}.
\begin{proposition}[\cite{hellmuth2013orthology}, Proposition 3]
  \label{prop:hellmuth2}
  A complete edge-colored graph is the graph representation of a symbolic
  ultrametric if and only if it does not contain a rainbow triangle and
  each monochromatic subgraph is a cograph.
\end{proposition}

Not every simple permutation graph is a cograph. For example, a path on
four vertices represents the permutation graph of~$(3,1,4,2)$ but not
a cograph.  
Intriguingly, all cographs are simple permutation
graphs~\cite{Bose1998}.  This, together with \Cref{prop:hellmuth2} and
\Cref{thm:characterization_all_k_PG}~(v) immediately implies
\begin{corollary}
  \label{coro:allSymbU_perm}
  The graph representation $G_{\delta}$ of a symbolic ultrametric
  $\delta$ is a complete edge-colored permutation graph.
\end{corollary}

Separable permutations can be represented by a so-called separating
  tree that reflects the structure of the permutation \cite{Bose1998} and
  can be characterized in terms of so-called permutation patterns.  For our
  purposes, the following result is of interest:
\begin{proposition}[\cite{Bose1998}]
  \label{pro:cograph_permutation2}
  Let $G$ be a simple permutation graph of $\pi$. Then, $\pi$ is
  separable if and only if $G$ is a cograph.
\end{proposition}	
It is easy to verify that every induced subgraph of a cograph is again a
cograph. This, together with \Cref{pro:cograph_permutation2},
\Cref{thm:characterization_all_k_PG}, and \Cref{prop:hellmuth2} implies
\begin{corollary}
  \label{coro:symb_is_perm}
  The following statements are equivalent:
  \begin{itemize}
  \item[i)] A complete edge-colored graph $G$ is the graph representation of a
    symbolic ultrametric.
  \item[ii)] $G$ does not contain a rainbow triangle and every monochromatic
    subgraph of $G$ is the simple permutation graph of a separable
    permutation.
  \end{itemize}
  Moreover, if $G$ is the graph representation of a symbolic ultrametric,
  then every monochromatic subgraph of each induced subgraph is the simple
  permutation graph of a separable permutation.
\end{corollary}


\section{Summary and Outlook}
\label{sec:conclusion}

In this paper, we characterized complete edge-colored permutation graphs,
that is, complete edge-colored graphs that can be decomposed into
permutation graphs. These graphs form a hereditary class, \ie, every
induced subgraph of a complete edge-colored permutation graph is again a
complete edge-colored permutation graph, possibly with fewer colors.
Moreover, we showed that complete edge-colored permutation graphs are
equivalent to those complete edge-colored graphs, \ie, symmetric
2-structures, that contain no rainbow triangle and whose monochromatic
subgraphs are simple permutation graphs.  Further characterizations in
terms of the modular decomposition of complete edge-colored graphs are
provided.  In particular, complete edge-colored permutation graphs are
characterized by the structure of the quotient graphs of their strong
(prime) modules.  Moreover, we showed that complete edge-colored
permutation graphs $G=(V,E)$ can be recognized in $\Mc{O}(|V|^2)$-time and
that both the labeling and the underlying permutations can be constructed
within the same time complexity.

As a by-product, we observed that the edge-coloring of a complete
edge-colored permutation graph is always a Gallai coloring.  In addition,
we have shown the close relationship of edge-colored permutation graphs to
symbolic ultrametrics, separable permutations, and cographs. In particular,
the class of edge-colored graphs representing symbolic ultrametrics is
strictly contained in the class of complete edge-colored permutation
graphs.
	
There are many open problems that are of immediate interest for future
research. The first obvious open problem is to generalize the results
presented here to $k$-edge-colored graphs that are not necessarily complete,
\eg, the graph illustrated in \Cref{fig:2vs3}~(a).  Second, we may consider
the converse problem of the one studied in this paper: Consider simple
unlabeled permutation graphs $G_1,\dots,G_k$. The question arises whether
there is a complete edge-colored permutation graph~$G$ such that every
monochromatic subgraph $G_{|i}$ of~$G$ is isomorphic to $G_i$,
$1\leq i\leq k$.  Third, one may also ask whether one can add colored edges
to a given (not necessarily complete) edge-colored graph such that the
resulting graph is a complete edge-colored permutation graph.

\subsection*{Acknowledgements}
  This work was funded in part by the German Research Foundation
  (DFG STA 850/49-1) and Federal Ministry of Education and Research (BMBF 031L0164C).

\bibliographystyle{plain}
\bibliography{literature}
\end{document}